\documentclass{amsart}
\usepackage{amssymb,latexsym,graphics,enumerate,color}
\usepackage{fullpage}
\usepackage{graphicx}
\usepackage[font=small,labelfont=bf]{caption}
\numberwithin{equation}{section}
\usepackage{bbm}
\usepackage[normalem]{ulem}

\usepackage{xcolor,cancel}

\newtheorem{theorem}{Theorem}[section]
\newtheorem{corollary}[theorem]{Corollary}
\newtheorem{lemma}[theorem]{Lemma}
\newtheorem{proposition}[theorem]{Proposition}

\newcommand{\la}{\langle}
\newcommand{\ra}{\rangle}

\renewcommand{\O}{{\mathcal{O}}}

\newcommand{\R}{{\mathbb{R}}}

\newcommand{\cd}{{\,\cdot\,}}

\newcommand{\tC}{{\tilde{C}}}

\newcommand{\uu}{{\underline{u}}}
\newcommand{\du}{{\partial_u}}
\newcommand{\duu}{{\partial_\uu}}
\newcommand{\duc}{{\partial_{u_c}}}
\newcommand{\duuc}{{\partial_{\uu_c}}}
\newcommand{\ttau}{{{\tilde{\tau}}}}
\newcommand{\tR}{{\tilde{R}}}

\newcommand{\lc}{\lesssim}

\begin{document}
\bibliographystyle{plain}

\title
{
The sharp lifespan for a system of multiple speed wave equations:
Radial case
}

\dedicatory{Dedicated to Professor Thomas C.~Sideris on the occasion
  of his 70th birthday}

\author{Marvin Koonce}
\author{Jason Metcalfe}
\address{Department of Mathematics, University of North Carolina, Chapel Hill}
\email{mkoonce@unc.edu, metcalfe@email.unc.edu}

\keywords{wave equations, local energy estimates, multiple speed,
  almost global existence}

\thanks{The authors gratefully acknowledge the partial support from
  NSF grants DMS-2054910 and DMS-2135998.  The results contained
  herein were developed as a part of the Undergraduate Honors Thesis
  of MK.  JM thanks Yizhou Gu and Xiao-Ming Porter who completed some
  calculations related to this material as a part of their
  undergraduate research.}

 \begin{abstract}
Ohta examined a system of multiple speed wave equations with small
initial data and demonstrated a finite time blowup.  We show, in the
radial case, that the same system exists almost globally with the same
lifespan as a lower bound.  To do this, we use 
integrated local energy estimate, $r^p$ weighted local energy estimates, the Morawetz estimate that results
from using the scaling vector field as a multiplier, and mixed speed
ghost weights. 
 \end{abstract}

\maketitle


\section{Introduction}
This article focuses on establishing the sharp lifespan, in the radial
case, for a multiple speed system of wave equations with small initial
data introduced in \cite{Ohta}.  Based on \cite{Lindblad}, we know
that there exists a constant $c$ so 
that quasilinear wave equations of the form
\begin{equation}\label{loworder}
  \begin{cases}
    \Box u = Q(u,\partial u, \partial^2 u),\quad (t,x)\in
    \R_+\times\R^3,\\
    u(0,\cd)=\varepsilon f,\quad \partial_t u(0,\cd)=\varepsilon g,
  \end{cases}\end{equation}
where $f, g\in C_c^\infty(\R^3)$ and $Q$ vanishes to second order at
the origin, have solutions on $[0, T_*]$ with $T_*\ge c/\varepsilon^2$ if
$\varepsilon$ is sufficiently small.  With the additional assumption
that $(\partial_u^2 Q)(0,0,0)=0$, which in essence rules out $u^2$
terms and leaves $u\partial u$ nonlinearities at the lowest order,
almost global existence $T_*\ge \exp(c/\varepsilon)$ was proved.
The latter result was partly extended to systems of equations in \cite{M-Rhoads}.

In the case of multiple speed systems of wave equations
\[\Box_{c_I} u^I = Q^I(\partial u, \partial^2 u),\quad (t,x)\in
  \R_+\times \R^3,\]
almost global existence was established in \cite{KlSid}.  Here
\[\Box_c = \partial_t^2-c^2\Delta\]
denotes the d'Alembertian at speed $c$.
Like in the
single speed case \cite{Klainerman}, \cite{Christodoulou}, if the
nonlinearity satisfies a null condition, then global existence can be
recovered.  In the semilinear case, if
\[Q^I(\partial u) = B^{I,\alpha\beta}_{JK} \partial_\alpha u^J
  \partial_\beta u^K + C(\partial u),\]
where $C$ vanishes to third order and repeated variables are
implicitly summed using the Einstein convention, the null condition is
only necessary when $c_I=c_J=c_K$.  See, e.g., \cite{Sideris-Tu},
\cite{Sideris}, \cite{Sogge}, \cite{mns2, mns1}.
The reason that it suffices to only have an assumption on the same
speed interactions is 
that solutions to the wave equation enjoy additional decay off of the light cone, and when there are differing speeds
one of the factors will be away from its light cone thus contributing more rapid decay.

For multiple speed analogs of \eqref{loworder} with $(\partial_u^2
Q)(0,0,0)=0$, one may wonder if global existence can be recovered provided no
quadratic nonlinear term in the equation for $u^I$ has factors that
are both at the same speed $c_I$.  More precisely, if we truncate to
quadratic level for semilinear equations
\[\Box_{c_I}u^I = A^{I,\alpha}_{JK} u^J \partial_\alpha u^K +
  B^{I,\alpha\beta}_{JK} \partial_\alpha u^J \partial_\beta u^K,\]
it may be reasonable to expect global existence provided that
\begin{equation}\label{msseparation}A^{I,\alpha}_{JK}=0=B^{I,\alpha\beta}_{JK},\quad\text{whenever }
  c_I=c_J=c_K.\end{equation}

In a somewhat surprising result, \cite{Ohta} demonstrated that
\eqref{msseparation} is not a sufficient condition for global
existence for sufficiently small initial data.  Indeed, for $c>1$, the
following system was considered.
\begin{equation}
  \label{ohta_system}
  \begin{cases}
    \Box v = w \partial_t v,\quad\quad (t,x)\in \R_+\times\R^3,\\
    \Box_c w = (\partial_t v)^2,\\
    v(0,\cd)=\varepsilon v_{(0)},\quad \partial_t v(0,\cd)=\varepsilon
    v_{(1)},\\
    w(0,\cd)=\varepsilon w_{(0)},\quad \partial_t w(0,\cd)=\varepsilon w_{(1)}.
  \end{cases}
\end{equation}
It was established that there are smooth, compactly supported initial
data $v_{(j)}, w_{(j)}$ and a constant $c_0$ so that the lifespan
$T_*$ satisfies $T_* \le \exp(c_0/\varepsilon^2)$.

The current study seeks to show the reverse inequality.  That is, for
any data $v_{(0)}, v_{(1)}, w_{(0)}, w_{(1)}$ that are smooth and
compactly supported, we seek to show that there is a constant $\tilde{c}$ so
that $T_\varepsilon \ge \exp(\tilde{c}/\varepsilon^2)$.
Without loss of generality, we take the
supports of the data to be contained within the unit ball $\{|x|\le 1\}$.  We
shall also use time translation symmetry and  take the initial data on the time slice $t=4$.
In the current article we consider only the radial case:
$v_{(j)}(x)=v_{(j)}(|x|)$, $w_{(j)}(x)=w_{(j)}(|x|)$.

Using the assumption of radial symmetry, we can reduce the question at
hand to a problem in $(1+1)$-dimensions.  Indeed, by conjugating, we have
\[\Box_c w(t,r) = r^{-1}(\partial_t^2-c^2 \partial_r^2)(rw).\]
If we set $V(t,r)=rv(t,r)$ and $W(t,r)=rw(t,r)$, we can instead seek
sufficiently regular solutions to the $(1+1)$-dimensional
initial-value boundary-value problem
\begin{equation}
  \label{ohta_1d_system}
  \begin{cases}
    \Box V = r^{-1} W \partial_t V,\quad\quad (t,r)\in \R_+\times\R_+,\\
    \Box_c W =r^{-1} (\partial_t V)^2,\\
    W(t,0)=V(t,0)=0,\quad\text{ for all } t,\\
    V(4,\cd)=\varepsilon V_{(0)},\quad \partial_t V(4,\cd)=\varepsilon
    V_{(1)},\\
    W(4,\cd)=\varepsilon W_{(0)},\quad \partial_t W(4,\cd)=\varepsilon W_{(1)}.
  \end{cases}
\end{equation}
Here (when applied to $V, W$) we understand $\Box_c = \partial_t^2 -
c^2\partial_r^2$ to be the $(1+1)$-dimensional d'Alembertian.

If we extend $V_{(j)}$, $W_{(j)}$ in an odd fashion,
$V_{(j)}(-r)=-V_{(j)}(r)$, $W_{(j)}(-r)=-W_{(j)}(r)$, 
then it is straightforward to check that $V, W$ extend oddly, and we
can instead seek to solve
\begin{equation}
  \label{ohta_1d_system_odd}
  \begin{cases}
    \Box V = x^{-1} W \partial_t V,\quad\quad (t,x)\in \R_+\times\R,\\
    \Box_c W =x^{-1} (\partial_t V)^2,\\
    V(4,\cd)=\varepsilon V_{(0)},\quad \partial_t V(4,\cd)=\varepsilon
    V_{(1)},\\
    W(4,\cd)=\varepsilon W_{(0)},\quad \partial_t W(4,\cd)=\varepsilon W_{(1)}.
  \end{cases}
\end{equation}

The main theorem is a statement of almost global existence for \eqref{ohta_1d_system_odd}.
\begin{theorem}\label{main_theorem} 
Suppose that $V_{(j)}, W_{(j)}\in C^\infty_c(\R)$ and
$V_{(j)}(x)=-V_{(j)}(-x), W_{(j)}(x)=-W_{(j)}(-x)$ for $j=0,1$. Then
there exist constants $\tilde{c}, \varepsilon_0 >0$ such that when
$0<\varepsilon < \varepsilon_0$, \eqref{ohta_1d_system_odd} has a unique
solution $(V,W)\in C^\infty([4,T_\varepsilon]\times\R),$ where  
\begin{equation} 
    \label{lifespan} 
    T_\varepsilon = \exp(\tilde{c}/\varepsilon^2).
  \end{equation}
\end{theorem}
As indicated above, Theorem 1.1 leads to the following corollary, showing that, in the radial case, the upper bound on the lifespan in \cite{Ohta} is sharp.
\begin{corollary}\label{main}
  Suppose that $v_{(j)}, w_{(j)}\in C^\infty_c(\R^3)$, and that
  $v_{(j)}(x)=v_{(j)}(|x|), w_{(j)}(x)=w_{(j)}(|x|)$ for $j=0,1$. Then
  there exist constants $\tilde{c}, \varepsilon_0 >0$ such that when
  $0<\varepsilon < \varepsilon_0$, \eqref{ohta_system} has a unique
  solution $(v,w)\in C^\infty([0,T_\varepsilon]\times\R^3)$ with
  $T_\varepsilon$ as in \eqref{lifespan}.
\end{corollary}

Our proof relies on the method of invariant vector fields with
adaptations to a restricted set of vector fields, which is
necessitated as the Lorentz boosts have an associated speed and only
commute with the d'Alembertian of the same speed.  This was pioneered
in, e.g., \cite{KlSid}, \cite{KSS, KSS3}, \cite{MS3, MS, ms_mathz},
\cite{mns2, mns1}.  The method is further simplified in the radial case as
the only relevant vector fields are derivatives and the scaling vector
field.  And in the study at hand, we will rely solely on the scaling
vector field to avoid issues with commuting derivatives with the
singular weight introduced in the one dimensional reduction.

Like \cite{KSS} and \cite{MS}, we will call upon a class of integrated
local energy estimates in order to prove long-time existence.  These
estimates go back to the seminal work \cite{Morawetz}.  In the
current setting, a number of variants are used.  These include the
$r^p$-weighted estimates of \cite{dafermos_rodnianski}, ghost weighted
estimates from \cite{Alinhac_ghostweight}, and integrated estimate
variants of \cite{Morawetz_scaling} where the scaling vector field is
used as a multiplier.

To obtain the requisite decay, rather
than using the classical Klainerman-Sobolev estimate \cite{klainermanSob}, which
introduces Lorentz 
boosts, or the weighted Sobolev estimate \cite{klainermanSob}, which provides decay in
$|x|$ but fails to capture the added decay off of the light cone, as
was done in \cite{KSS} and \cite{MS}, we
will instead use a class of space-time Klainerman-Sobolev estimates of
\cite{MTT}.  The space-time nature of these decay bounds meshes well
with the integrated local energy estimates. 

In the estimates that we prove, we allow for ghost weights that are
associated to a different speed than the equation.  This works in the
non-radial case (see the forthcoming work \cite{bechtold_m}) provided that the speed within the ghost weight
exceeds that of the speed of the equation.  In the radial case,
however, no restriction on the speeds is needed.  This is the
primary place where we rely upon the radiality assumption.  We anticipate
examining the general case in a future study.

After posting this result, we learned of the paper \cite{katayama}.
Corollary \ref{main} is a direct consequence of the main result of
\cite{katayama}.  We thank Mengyun Liu and Chengbo Wang for alerting
us to this.  Our physical space techniques are quite distinct from the
approach of \cite{katayama}, which is rooted in the fundamental
solution.  We believe the current methods present a path toward
dropping the assumption of radial symmetry and to the study of these
equations in the presence of background geometry.  As such, we believe
that the proof contained herein is of independent interest.

\medskip
\section{Notation and Decay Estimates}

For $x\in \R$, we denote $r=|x|$.
We will use the standard null coordinates
\[u=t-r,\quad \uu=t+r,\quad \partial_u =
  \frac{1}{2}(\partial_t-\partial_r),\quad \partial_{\uu} =
  \frac{1}{2}(\partial_t+\partial_r).\]
The above correspond to speed $1$, but at wave speed $c$, we will
instead have
\[u_c = ct-r,\quad \uu_c=ct+r,\quad \partial_{u_c} =
  \frac{1}{2c}(\partial_t-c\partial_r),\quad \partial_{\uu_c} = \frac{1}{2c}(\partial_t+c\partial_r).\]
We denote the scaling vector field by
\begin{equation}\label{scaling}S=t\partial_t+r\partial_r = u\partial_u + \uu\partial_\uu =
  u_c\partial_{u_c} +\uu_c\partial_{\uu_c}.
\end{equation}
It will be important for later purposes to note
\[[\duc, S]=\duc,\quad [\duuc,S]=\duuc, \quad [\Box, S] = 2\Box.\]
For $k\in \mathbb{N}$, we use the notation
\[|S^{\le k} w| = \sum_0^k |S^j w|.\]

The space-time Klainerman-Sobolev-type estimates of \cite{MTT} will be
the principal source of decay.  See also \cite{M-Stewart}.  The only
difference herein is allowing for $c\neq 1$
and simplifications that result from the $(1+1)$-dimensional regime.  As
such, we will be brief in the presentation.

If at speed $c$, the compactly supported data (at $t=4$) are taken to
be supported
in the unit ball, the components of the solution will be supported in
$C^c=\{(t,x)\,:\,t\in [4,T_\varepsilon],\,r\le ct-(4c-1)\}$.
We now dyadically decompose $C^c$ in both $t$ and in
either $r$ or $u_c$ depending on the proximity to the light cone.
We first decompose in $t$ and set
\[C^c_\tau = \{(t,r)\in \R_+\times\R_+\,:\, t\in [4,T_\varepsilon]\cap
  [\tau, 2\tau],\, r\le ct-(4c-1)\}.\]
We further break into
\[C^{c,R=1}_\tau= C^c_\tau \cap \{r\le 2\},\quad C^{c,R}_\tau = C^c_\tau
  \cap \{R\le r\le 2R\} \text{ when $1<R$},\]
and
\[
 C^{c,U_c}_\tau=
  C^c_\tau\cap \{U_c\le ct-r\le 2U_c\} \text{ when $1< U_c.$}\]
We finally set
\[C^{c,c\tau/2}_\tau = C^c_\tau \cap \{ct-r\ge c\tau/2\} \cap \{r\ge
  c\tau/2\}.\]
Throughout $\tau, R, U_c, U_1:=U$ will be understood to range over
dyadic values.
This gives
\begin{equation}\label{decomp_c}
  C^c_\tau = \Bigl(\bigcup_{1\le R\le c\tau/4} C^{c,R}_\tau\Bigr) \cup
  \Bigl(\bigcup_{2\le U_c\le c\tau/4} C^{c,U_c}_\tau\Bigr)\cup
  C^{c,c\tau/2}_\tau.
\end{equation}
We will let $\tC^{c,R}_\tau$ and $\tC^{c,U_c}_\tau$ denote slight
enlargements (in both scales) to allow for tails of cutoff functions.
On the components of the decomposition (and their enlargements) we
have:
\[\la r \ra \approx R, \quad t\approx \tau, \quad u_c \approx \tau\quad \text{ on }
  C^{c,R}_\tau,\, 1\le R\le c\tau/4,\]
\[r\approx \tau,\quad t\approx \tau, \quad \la u_c\ra\approx U_c\quad \text{ on }
  C^{c,U_c}_\tau,\, 2\le U_c\le c\tau/4.\]
The remaining region $C^{c,c\tau/2}_\tau$ may be thought of as
either $R=c\tau/2$ or $U_c=c\tau/2$.

The following are space-time analogs of the Klainerman-Sobolev
estimates:
\begin{lemma}[\cite{MTT}]\label{lemma_wklsob} 
  Suppose $W\in C^2([4,T_\varepsilon]\times \R)$ is an odd function.
  If $c>0$, $\tau\ge 4$, $1\le R\le c \tau/2$, $2\le U_c\le c\tau/4$,
  then
  \begin{align}
   \label{WR}\|r^{-\frac{1}{2}}W\|_{L^\infty_tL^\infty_r(C^{c,R}_\tau)} &\lesssim
    \frac{1}{\tau^{\frac{1}{2}} R} \|S^{\le 1}
    W\|_{L^2_tL^2_r(\tC^{c,R}_\tau)} +
    \frac{1}{\tau^{\frac{1}{2}}} \|\partial_r S^{\le 1}
    W\|_{L^2_t L^2_r(\tC^{c,R}_\tau)},\\
   \label{WU} \|W\|_{L^\infty_t L^\infty_r(C^{c,U_c}_\tau)}&\lesssim
                                                   \frac{1}{\tau^{\frac{1}{2}}U_c^{\frac{1}{2}}}
                                                   \|S^{\le 1}
                                                   W\|_{L^2_tL^2_r(\tC^{c,U_c}_\tau)}
                                                   +
                                                   \frac{U_c^{\frac{1}{2}}}{\tau^{\frac{1}{2}}}
                                                   \|\partial_r S^{\le 1} W\|_{L^2_tL^2_r(\tC^{c,U_c}_\tau)}.
  \end{align}
\end{lemma}

Here and throughout, $L^2_r$ is simply the $1$-dimensional norm:
\[\|f\|_{L^2_r}^2 = \int_0^\infty |f(r)|^2\,dr.\]

\begin{proof}
On $C^{c, U_c}_\tau$ regions and on $C^{c, R}_\tau$ regions for $R
>1$, the result follows from \cite{MTT} in one spatial dimension. On
the $C^{c, R=1}_\tau$ regions, however, a different change of
variables is required to avoid picking up vector fields other than
$S$.  Some care is also taken to assist with the singular behavior at
$r=0$ that was introduced in the reduction to one dimension.

Let $\beta:[0,\infty)\to [0,1]$ be a smooth cutoff function such that
$\beta(y)=1$ for $y\in [1,2]$ and $\beta(y)=0$ for $y\in
[0,1-\delta]\cup [2+\delta,\infty)$ where $0<\delta\ll 1$.  We examine $\beta(t/\tau) W(t,r)$.  We first change variables to
$t=e^s$ and $r=\rho e^s$.  Applying
the Fundamental Theorem of Calculus in $s$ and $\rho$ (relying up on the
fact that $S^{\le 1} W$ is odd and hence vanishes at $r=0$), we have 
\[
    |\beta(e^s/\tau)W(e^s, \rho e^s)|
    \lc
    \int_0^\rho \int_{-\infty}^{\log t} |\partial_\zeta
    \partial_s(\beta(e^s/\tau) W(e^s, \zeta e^s))|\,ds\,d\zeta.
 \]
 Since $\partial_s (W(e^s, \zeta e^s)) = (SW)(e^s,\zeta e^s)$, upon
 converting back to $(t,r)$-coordinates, we have
 \begin{multline}\label{r1}
  | \beta(t/\tau) W(t,r)| \lesssim
 \tau^{-2}  \int\int_{[\tau(1-\delta),t]\times [0,r]} |S^{\le 1}
   W(t,z)|\,dz\,dt
   \\+ \tau^{-1} \int\int_{[\tau(1-\delta),t]\times [0,r]}
   |(\partial_r S^{\le 1} W)(t,z)|\,dz\,dt.
 \end{multline}
 Applying the Schwarz inequality to the right side and multiplying
 through by $r^{-\frac{1}{2}}$ then yields
\[ \|r^{-\frac{1}{2}} W\|_{L^\infty_tL^\infty_r(C^{c,R=1}_\tau)}
 \lesssim \frac{1}{\tau^{\frac{3}{2}}} \|S^{\le 1}W\|_{L^2_tL^2_r(\tC^{c,R=1}_\tau)}
 + \frac{1}{\tau^{\frac{1}{2}}} \|\partial_r S^{\le 1}
 W\|_{L^2_tL^2_r(\tC^{c,R=1}_\tau)},\]
which is stronger than \eqref{WR} when $R=1$.
\end{proof}

By writing $\partial_r = \partial_{\uu_c} - \partial_{u_c}$ and using \eqref{scaling}, we can
obtain the following corollary.  See, e.g., \cite{M-Stewart}.
\begin{corollary}\label{cor_wklsob}
  Suppose $W\in C^2([4,T_\varepsilon]\times\R)$ is an odd function.
   If $c>0$, $\tau\ge 4$, $1\le R\le c \tau/2$, $2\le U_c\le c\tau/4$,
  then
  \begin{align}\label{wcrt}
   \|r^{-\frac{1}{2}}W\|_{L^\infty_tL^\infty_r(C^{c,R}_\tau)} &\lesssim
    \frac{1}{\tau^{\frac{1}{2}} R} \|S^{\le 2}
    W\|_{L^2_tL^2_r(\tC^{c,R}_\tau)} +
    \frac{1}{\tau^{\frac{1}{2}}} \|\duuc S^{\le 1}
    W\|_{L^2_t L^2_r(\tC^{c,R}_\tau)},\\
    \|W\|_{L^\infty_t L^\infty_r(C^{c,U_c}_\tau)}&\lesssim
                                                   \frac{1}{\tau^{\frac{1}{2}}U_c^{\frac{1}{2}}}
                                                   \|S^{\le 2}
                                                   W\|_{L^2_tL^2_r(\tC^{c,U_c}_\tau)}
                                                   +
                                                   \frac{\tau^{\frac{1}{2}}}{U_c^{\frac{1}{2}}}
                                                   \|\duuc S^{\le
                                                   1}
                                                   W\|_{L^2_tL^2_r(\tC^{c,U_c}_\tau)}. \label{wcut}
  \end{align}
  \end{corollary}

Using \eqref{scaling}, we also obtain the following bounds on $\partial V:=(\partial_t, \partial_r)V$,
which appeared previously in \cite{M-Stewart}.  They are space-time
variants of estimates originally from \cite{KlSid} and \cite{Sideris-Thomases}.

\begin{lemma} \label{vbounds} Suppose $V\in C^3([4,T_\varepsilon]\times \R)$ is an
  odd function.
If $\tau\ge 4$, $1\le R\le \tau/2$, and $2\le U\le \tau/4$,
  then
  \begin{align}\label{vcrt}
   \|V\|_{L^\infty_tL^\infty_r(C^{1,R}_\tau)} &\lesssim
    \frac{1}{\tau^{\frac{1}{2}} R^{\frac{1}{2}}} \|S^{\le 2}
    V\|_{L^2_tL^2_r(\tC^{1,R}_\tau)} +
    \frac{R^{\frac{1}{4}}}{\tau^{\frac{1}{2}}} \|r^{\frac{1}{4}}\duu S^{\le 1}
    V\|_{L^2_t L^2_r(\tC^{1,R}_\tau)},\\
    \|V\|_{L^\infty_t L^\infty_r(C^{1,U}_\tau)}&\lesssim
                                                   \frac{1}{\tau^{\frac{1}{2}}U^{\frac{1}{2}}}
                                                   \|S^{\le 2}
                                                   V\|_{L^2_tL^2_r(\tC^{1,U}_\tau)}
                                                   +
                                                   \frac{\tau^{\frac{1}{2}}}{U^{\frac{1}{2}}}
                                                   \|\duu S^{\le
                                                   1}
                                                 V\|_{L^2_tL^2_r(\tC^{1,U}_\tau)} \label{vcut},\\
    \|\partial V\|_{L^\infty_t L_r^\infty(C^{1,R}_\tau)} &\lesssim
                                                       \frac{1}{\tau^{\frac{1}{2}}R^{\frac{1}{2}}}
                                                       \|\partial
                                                       S^{\le 2}
                                                           V\|_{L^2_tL^2_r(\tC^{1,R}_\tau)}
                                                           +
                                                           \frac{R^{\frac{1}{4}}}{\tau^{\frac{1}{2}}}
                                                           \|r^{\frac{1}{4}}\Box
                                                           S^{\le 1}
                                                           V\|_{L^2_tL^2_r(\tC^{1,R}_\tau)},\label{kscrt}\\
    \|\partial V\|_{L^\infty_tL^\infty_r(C^{1,U}_\tau)} &\lesssim
                                                            \frac{1}{
                                                            \tau^{\frac{1}{2}}U^{\frac{1}{2}}}
                                                            \|\partial
                                                            S^{\le 2}
                                                            V\|_{L^2_tL^2_r(\tC^{1,U}_\tau)}
                                                            +
                                                            \frac{\tau^{\frac{1}{2}}}{U^{\frac{1}{2}}
                                                            }
                                                            \|\Box
                                                            S^{\le 1}
                                                          V\|_{L^2_tL^2_r(\tC^{1,U}_\tau)}.\label{kscut}
  \end{align}
\end{lemma}

\begin{proof}
 Outside of the $R=1$ case, the above follow from \cite{MTT} and
 \cite{M-Stewart}.  When $R=1$, we argue as in \eqref{r1}.  At this
 point, a different application of the Schwarz inequality yields
 \[
   \beta(t/\tau) V(t,r)\lesssim
   \frac{1}{\tau^{\frac{3}{2}}} \|S^{\le 1}
   V\|_{L^2_tL^2_r(\tC^{1,R=1}_\tau)} +
   \frac{r^{\frac{1}{4}}}{\tau^{\frac{1}{2}}} \|r^{\frac{1}{4}}
   \partial_r S^{\le 1} V\|_{L^2_tL^2_r(\tC^{1,R=1}_\tau)}.\]
 Using \eqref{scaling} then allows us to recover \eqref{vcrt}.  A
 subsequent application of \eqref{scaling}, as in \cite{M-Stewart}, yields \eqref{kscrt}.
\end{proof}

It will be of utmost importance to track the availability to sum over
the dyadic ranges $R, \tau, U, U_c$.  To this end, we shall use
notation such as
\[\|W\|^2_{\ell^2_\tau \ell^2_{U\le \tau/4} L^2_tL^2_r(C^{1,U}_\tau)}
  =  \sum_{\tau} \sum_{U\le \tau/4} \|W\|^2_{L^2_t
    L^2_r(C^{1,U}_\tau)},\quad
\|W\|^2_{\ell^\infty_U \ell^2_{\tau} L^2_tL^2_r(C^{1,U}_\tau)}
  =  \sup_{U\ge 1} \sum_{\tau} \|W\|^2_{L^2_t
    L^2_r(C^{1,U}_\tau)}.
\]
Other variants, such as over $C^{1,R}_\tau$, $C^{c, U_c}_\tau$, etc.,
are similarly defined.

\medskip
\section{Local Energy Estimates}
In this section we shall gather the energy estimates that will be used
in the sequel.  These are variants of the integrated local energy
estimates, which originated in \cite{Morawetz}.  See also
\cite{KSS}, \cite{Sterbenz}, \cite{MS}, \cite{MST} for subsequent
generalizations and applications to nonlinear equations.  These ideas
are combined with an integrated variant of the classical estimate
obtained by using the scaling vector field as a multiplier as was done
in
\cite{Morawetz_scaling}.  We also utilize a ghost weight as introduced
in \cite{Alinhac_ghostweight}.  As in \cite{bechtold_m}, we use ghost
weights where the speed does not necessarily coincide with the given
equation.  For the speed $c$ component, our principal estimate is an
$r^p$-weighted (with $p=1$) integrated local energy estimate from
\cite{dafermos_rodnianski}.  It is combined with a ghost weight as was
done in \cite{M-Rhoads}.  We pair this with a Hardy inequality that
takes advantage of the multiple speeds.

For weighting functions, we shall use, for a parameter $\theta\ge 1$,
\begin{equation}
  \label{weights}
\sigma_\theta(y) = \frac{y}{|y|+\theta}, \quad \sigma_\theta'(y)=\frac{\theta}{(|y|+\theta)^2},
\end{equation}
which is bounded and $C^1$.  Moreover, we note that
   \begin{equation}
     \label{weights_properties}
\sigma'_\theta(y)\gtrsim \theta^{-1} \quad\text{ on $\{\la y\ra \approx \theta\}$}.
 \end{equation}
In order to help control the singularity at $r=0$ that results from
the reduction to one-dimension, we also note
\begin{equation}
  \label{weights_properties2}
  \frac{d}{dr} (\sigma_R(r))^{\frac{1}{2}} \gtrsim r^{-\frac{1}{2}}
  R^{-\frac{1}{2}} \quad\text{ on $\{\la r\ra\approx R\}$.}
\end{equation}
Using $(\sigma_R(r))^{\delta}$ with $0<\delta<1$ as a weight in order
to gain added control at $r=0$ has appeared previously
in, e.g., \cite{Hidano-Wang-Yokoyama}.

We shall use the following proposition when $p=0$ and $p=1$.  In the former
case, this corresponds to the classical integrated local energy
estimates and variants that are available using the ghost weight.  For
$p=1$, these are instead variants of \cite{Morawetz_scaling} where
$S=t\partial_t+r\partial_r$ is used as a multiplier.  The fact that no
upper bound on $p$ is necessary is a consequence of the radiality
assumption.  The angular terms necessitate $p\le 2$ in more general
cases.  We start with a lemma, which indicates that the ``good''
derivative portion of the calculation works independently.
 
\begin{lemma}\label{lemma_good_v}
  Suppose that $p\geq0$, $V\in C^2([4,T]\times\R)$, and that there exists
  $\tilde{R}>0$ so that $V(t,x)\equiv 0$ whenever $|x|>\tilde{R}$.
Then
  \begin{multline}
    \label{good_v}
    \|r^{-\frac{1}{4}}\la r\ra^{-\frac{1}{4}}\la \uu\ra^{\frac{p}{2}}\duu
    V\|^2_{\ell^\infty_R \ell^2_\tau L^2_tL^2_r(C^{1,R}_\tau)} + \|\la
    u\ra^{-\frac{1}{2}} \la \uu\ra^{\frac{p}{2}}\duu V\|^2_{\ell^\infty_U
      \ell^2_\tau L^2_tL^2_r(C^{1,U}_\tau)}
    + \|\la
    u_c\ra^{-\frac{1}{2}} \la \uu\ra^{\frac{p}{2}}\duu V\|^2_{\ell^\infty_{U_c}
     \ell^2_\tau L^2_tL^2_r(C^{c,U_c}_\tau)}
 \\\lesssim \|\la r\ra^{\frac{p}{2}} \duu V(4,\cd)\|^2_{L^2_r} +
   \int_4^T\int_0^\infty \la\uu\ra^p |\Box V| |\duu V| \,dr\,dt.
  \end{multline}
\end{lemma}
\begin{proof}
  We consider
  \begin{multline}
    \label{good_v_multiplier}
    \int_4^T\int_0^\infty e^{(\sigma_R(r))^{\frac{1}{2}}} 
    e^{-\sigma_U(t-r)} e^{-\sigma_{U_c}(ct-r)}(1+t+r)^p
   \Box V(t,r)\,
    (\partial_t+\partial_r)V(t,r)\,dr\,dt
    \\=\frac{1}{2}\int_4^T\int_0^\infty
    e^{(\sigma_R(r))^{\frac{1}{2}}} e^{-\sigma_U(t-r)} e^{-\sigma_{U_c}(ct-r)}
    (1+t+r)^p
    (\partial_t-\partial_r)\Bigl((\partial_t+\partial_r)V\Bigr)^2\,dr\,dt. 
  \end{multline}
  Integrating by parts gives that this is equivalent to
  \begin{multline}\label{good_v_ibp}
  \frac{1}{2}  \int_0^\infty e^{(\sigma_R(r))^{\frac{1}{2}}}
   e^{-\sigma_U(T-r)}e^{-\sigma_{U_c}(cT-r)} (1+T+r)^p
    \Bigl((\partial_t+\partial_r)V(T,r)\Bigr)^2 \,dr 
    \\- \frac{1}{2} \int_0^\infty e^{(\sigma_R(r))^{\frac{1}{2}}}
   e^{-\sigma_U(4-r)}e^{-\sigma_{U_c}(4c-r)} (5+r)^p
    \Bigl((\partial_t+\partial_r)V(4,r)\Bigr)^2 \,dr
\\+ \frac{1}{2}\int_4^T e^{-\sigma_U(t)} e^{-\sigma_{U_c}(ct)} (1+t)^p
(\partial_r V(t,0))^2\,dt
\\+\frac{1}{2}\int_4^T\int_0^\infty  \frac{d}{dr}(\sigma_R(r))^{\frac{1}{2}}
  e^{(\sigma_R(r))^{\frac{1}{2}}} e^{-\sigma_U(t-r)} e^{-\sigma_{U_c}(ct-r)} (1+t+r)^p
  \Bigl((\partial_t+\partial_r)V\Bigr)^2\,dr\,dt
  \\+\int_4^T\int_0^\infty  e^{(\sigma_R(r))^{\frac{1}{2}}} 
  \sigma'_U(t-r)e^{-\sigma_U(t-r)}e^{-\sigma_{U_c}(ct-r)} (1+t+r)^p
  \Bigl((\partial_t+\partial_r)V\Bigr)^2\,dr\,dt
  \\+\frac{c+1}{2}\int_4^T\int_0^\infty  e^{(\sigma_R(r))^{\frac{1}{2}}} 
  e^{-\sigma_U(t-r)}\sigma'_{U_c}(ct-r)e^{-\sigma_{U_c}(ct-r)} (1+t+r)^p
  \Bigl((\partial_t+\partial_r)V\Bigr)^2\,dr\,dt.
\end{multline}
We drop the non-negative first and third terms.  We subsequently can
restrict the range of the fourth, fifth, and sixth terms so that
\eqref{weights_properties} can be applied.  We also use that
$\sigma_\theta$ is bounded uniformly in $\theta$.  For example,
\begin{multline*}
  \int_4^T\int_0^\infty  e^{(\sigma_R(r))^{\frac{1}{2}}} 
  \sigma'_U(t-r)e^{-\sigma_U(t-r)} e^{-\sigma_{U_c}(ct-r)} (1+t+r)^p
  \Bigl((\partial_t+\partial_r)V\Bigr)^2\,dr\,dt
\\  \gtrsim U^{-1} \int\int_{\{\la t-r\ra\approx U\}} (1+t+r)^p \Bigl((\partial_t+\partial_r)V\Bigr)^2\,dr\,dt.
\end{multline*}
Taking appropriate supremums in $R, U$ using the boundedness of
$\sigma_\theta$ then yields \eqref{good_v}.
\end{proof}

While the previous lemma worked independently, to get similar control
on the remaining derivatives, we must examine both the ``good''
derivative $\partial_t+\partial_r$ and the ``bad'' derivative
$\partial_t-\partial_r$ in unison.
In the radial case, this results from the $r=0$ boundary behavior.  In
more general situations, there is interaction amongst the angular
behavior as well.

\begin{proposition}\label{prop_leghost_v}
  Suppose that $p\geq0$, $V\in C^2([4,T]\times \R)$, and that there exists
  $\tilde{R}>0$ so that $V(t,x)\equiv 0$ whenever $|x|>\tilde{R}$.
Then 
  \begin{multline}
    \label{leghost_v}
    \|r^{-\frac{1}{4}}\la r\ra^{-\frac{1}{4}}\la \uu\ra^{\frac{p}{2}} \duu
 V\|^2_{\ell^\infty_R \ell^2_\tau L^2_tL^2_r(C^{1,R}_\tau)}
+ \|r^{-\frac{1}{4}}\la r\ra^{-\frac{1}{4}} \la u\ra^{\frac{p}{2}}\du
V\|^2_{\ell^\infty_R \ell^2_\tau L^2_tL^2_r(C^{1,R}_\tau)}
   \\ + \|\la u_c\ra^{-\frac{1}{2}} \la\uu\ra^{\frac{p}{2}}\duu
    V\|^2_{\ell^\infty_{U_c}\ell^2_\tau L^2_tL^2_r(C_\tau^{c,U_c})}
 + \|\la u_c \ra^{-\frac{1}{2}}\la u\ra^{\frac{p}{2}}\du
 V\|^2_{\ell^\infty_{U_c} \ell^2_\tau L^2_tL^2_r(C^{c,U_c}_\tau)}
\\ + \|\la
   u\ra^{-\frac{1}{2}} \la\uu\ra^{\frac{p}{2}}\duu V\|^2_{\ell^\infty_U
     \ell^2_\tau L^2_tL^2_r(C^{1,U}_\tau)} 
  \lesssim \|\la r\ra^{\frac{p}{2}}\partial V(4,\cd)\|^2_{L^2_r} 
  \\ +\int_4^T\int_0^\infty \la \uu\ra^p |\tilde{\Box} V| |\duu V| \,dr\,dt
   +\int_4^T\int_0^\infty \la u\ra^p |\tilde{\Box} V| |\du V| \,dr\,dt.
  \end{multline}
\end{proposition}

\begin{proof}
  We argue much like in the preceding proof but instead examine
  \begin{multline}
    \label{bad_v_multiplier}
\int_4^T\int_0^\infty e^{-(\sigma_R(r))^{\frac{1}{2}}} e^{-\sigma_{U_c}(ct-r)} (1+|t-r|)^p
    \Box V(t,r)\,
    (\partial_t-\partial_r)V(t,r)\,dr\,dt
    \\=\frac{1}{2}\int_4^T\int_0^\infty
    e^{-(\sigma_R(r))^{\frac{1}{2}}} e^{-\sigma_{U_c}(ct-r)} (1+|t-r|)^p (\partial_t+\partial_r)\Bigl((\partial_t-\partial_r)V\Bigr)^2\,dr\,dt.
  \end{multline}
  Integrating by parts shows that this is equal to
  \begin{multline}
    \label{bad_v_ibp}
\frac{1}{2}\int_0^\infty e^{-(\sigma_R(r))^{\frac{1}{2}}} e^{-\sigma_{U_c}(cT-r)} (1+|T-r|)^p
\Bigl((\partial_t-\partial_r)V(T,r)\Bigr)^2\,dr
\\-\frac{1}{2}\int_0^\infty e^{-(\sigma_R(r))^{\frac{1}{2}}}e^{-\sigma_{U_c}(4c-r)} (1+|4-r|)^p
\Bigl((\partial_t-\partial_r)V(4,r)\Bigr)^2\,dr
-\frac{1}{2}\int_4^T e^{-\sigma_{U_c}(ct)}
(1+t)^p(\partial_r V(t,0))^2\,dt
\\+\frac{1}{2}\int_4^T \int_0^\infty \frac{d}{dr}(\sigma_R(r))^{\frac{1}{2}} e^{-(\sigma_R(r))^{\frac{1}{2}}}e^{-\sigma_{U_c}(ct-r)}
(1+|t-r|)^p\Bigl((\partial_t-\partial_r)V\Bigr)^2\,dr\,dt
\\+\frac{c-1}{2}\int_4^T \int_0^\infty e^{-(\sigma_R(r))^{\frac{1}{2}}}\sigma'_{U_c}(ct-r)e^{-\sigma_{U_c}(ct-r)}
(1+|t-r|)^p\Bigl((\partial_t-\partial_r)V\Bigr)^2\,dr\,dt.
  \end{multline}
If it were not for the $r=0$ boundary term, we could argue as above to obtain our
estimate.  Here we instead must sum \eqref{good_v_ibp} and
\eqref{bad_v_ibp}.  The former provides the desired control in the
$r=0$ boundary term.  Then using \eqref{weights} and
\eqref{weights_properties} as in the preceding lemma
concludes the proof.
\end{proof}

We next proceed with the speed $c$ estimates.  This first estimate
combines the $r^p$ weighting of \cite{dafermos_rodnianski} with the
ghost weighting of \cite{Alinhac_ghostweight} as was done in
\cite{M-Rhoads}.  Here, however, we use a ghost weight at speed $1$
despite working with a solution to a speed $c$ wave equation.

\begin{proposition}\label{lemma_rp_w}
  Suppose that $c>0$, $W\in C^2([4,T]\times \R)$, and that there exists
  $\tilde{R}>0$ so that $W(t,x)\equiv 0$ whenever $|x|>\tilde{R}$.  Then
  \begin{multline}
    \label{rp_w}
    \|\duuc W\|^2_{\ell^2_\tau \ell^2_R L^2_tL^2_r(C^{1,R}_\tau)}
    + \|\la u\ra^{-\frac{1}{2}} r^{\frac{1}{2}} \duuc
    W\|^2_{\ell^\infty_U\ell^2_\tau L^2_tL^2_r(C^{1,U}_\tau)}
 \\   \lesssim \|r^{\frac{1}{2}} \duuc W(4,\cd)\|^2_{L^2_r}
+ \int_4^T\int_0^\infty r |\Box_c W| |\duuc W|\,dr\,dt.
  \end{multline}
\end{proposition}

\begin{proof}
  We now start with
\[
   \int_4^T\int_0^\infty r e^{-\sigma_U(t-r)}
   \Box_c W(t,r) (\partial_t+c\partial_r)W(t,r)\,dr\,dt
    =\frac{1}{2}\int_4^T\int_0^\infty r e^{-\sigma_U(t-r)}
    (\partial_t-c\partial_r)\Bigl((\partial_t+c\partial_r)W \Bigr)^2\,dr\,dt,
 \]
  which upon integrating by parts is seen to be the same as
  \begin{multline*}
  \frac{1}{2} \int_0^\infty r e^{-\sigma_U(T-r)}
    \Bigl((\partial_t+c\partial_r)W(T,r)\Bigr)^2\,dr
    -    \frac{1}{2}\int_0^\infty r e^{-\sigma_U(4-r)}
    \Bigl((\partial_t+c\partial_r)W(4,r)\Bigr)^2\,dr
    \\+\frac{c}{2}\int_4^T\int_0^\infty e^{-\sigma_U(t-r)}
    \Bigl((\partial_t+c\partial_r)W\Bigr)^2\,dr\,dt
    +\frac{1+c}{2} \int_4^T\int_0^\infty r\sigma_U'(t-r)
    e^{-\sigma_U(t-r)} \Bigl((\partial_t+c\partial_r)W\Bigr)^2\,dr\,dt.
  \end{multline*}
  After noting the first term is non-negative, we may
  drop it.  The proof is then completed by restricting the range of the last term so that
  \eqref{weights_properties} may be applied and using the boundedness
  of $\sigma_\theta$.
\end{proof}

The preceding lemma will be paired with two Hardy inequalities as $W$
appears in our nonlinearity without a derivative.  The first is a
space-time variant of the standard Hardy inequality.  See, e.g., \cite{M-Rhoads}.

\begin{lemma}\label{lemma_hardy}
 Suppose that $c>0$, $W\in C^1([4,T]\times \R)$ is an odd function, and
that there exists $\tilde{R}>0$ so that $W(t,x)\equiv 0$ whenever $|x|>\tilde{R}$.
Then
\begin{equation}\label{hardy}
  \| r^{-1} W\|_{L^2_tL^2_x([4,T]\times\R)}
  \lesssim \|r^{-\frac{1}{2}} W(4,\cd)\|_{L^2_x} + \|\duuc W\|_{L^2_tL^2_x([4,T]\times\R)}.
\end{equation}
\end{lemma}

\begin{proof}
  We write
\[  
     \int_4^T\int r^{-2} W^2\,dx\,dt =
  -\frac{2}{c}\int_4^T \int_0^{\infty} (\partial_t+c\partial_r)
  r^{-1} \cdot W^2\,dr\,dt.
  \]
As $W$ is $C^1$ and odd, the Mean Value Theorem gives that $W=\O(r)$
near $r=0$.  Thus, these integrals are well-defined and $r^{-1}
W^2(t,r)\to 0$ as $r\to 0$.
  Integrating by parts reveals
\begin{multline}\label{Hardy1}
  \int_4^T\int r^{-2} W^2\,dx\,dt +\frac{2}{c}
   \int_0^\infty r^{-1} (W(T,r))^2 \,dr 
   \\= \frac{2}{c}
   \int_0^\infty r^{-1}(W(4,r))^2 \, dr +
   \frac{4}{c}\int_4^T\int_0^\infty
   r^{-1}W(\partial_t+c\partial_r) W\,dr\,dt.
 \end{multline}
 The Schwarz inequality shows that
 \begin{align*}
   \frac{4}{c}\int_4^T\int_0^\infty
   r^{-1}W(\partial_t+c\partial_r) W\,dr\,dt &\lesssim
                                                         \|r^{-1}
                                                         W\|_{L^2_tL^2_x([4,T]\times\R)}
                                                         \|\duuc
                                                         W\|_{L^2_tL^2_x([4,T]\times\R)}\\
   &\le \frac{1}{2} \|r^{-1}  W\|^2_{L^2_tL^2_x([4,T]\times\R)}
     + C \|\duuc  W\|^2_{L^2_tL^2_x([4,T]\times\R)}. 
 \end{align*}
 Plugging this into \eqref{Hardy1} and absorbing the $\|r^{-1}
 W\|_{L^2_tL^2_x}$ term back into the left side yields the desired estimate.
\end{proof}

The second Hardy inequality that we rely upon takes advantage of the
multiple speed structure.
\begin{lemma}\label{lemma_hardy_mixed}
  Suppose $c>1$, $W\in C^1([4,T]\times \R)$ is an odd function, and that
there exists $\tilde{R}>0$ so that $W(t,x)\equiv 0$ when $|x|>\tilde{R}$.
Then
  \begin{equation}
    \label{hardy_mixed}
\| \la u\ra^{-\frac{1}{2}} r^{-\frac{1}{2}} W
\|_{\ell^\infty_U \ell^2_\tau L^2_tL^2_x(C^{1,U}_\tau)}
    \\\lesssim \|r^{-\frac{1}{2}} W(4,\cd)\|_{L^2_x}
    + \|\duuc W\|_{L^2_tL^2_x([4,T]\times\R)}.
  \end{equation}
\end{lemma}

\begin{proof}
  The argument here is similar to the preceding, but we take
  advantage of the difference in speeds between the weight $\la
  u\ra^{-1}$ and the speed $c$ of the equation.  To this end, using
  \eqref{weights_properties}, we observe
\[
    \iint_{\{\la t-r\ra \approx U\}} \la u\ra^{-1} r^{-1} W^2 \,dx\,dt
    \le \frac{2}{1-c} \int_4^T\int_0^\infty r^{-1}(\partial_t+c\partial_r)
    (e^{\sigma_U(t-r)}) \cdot W^2\,dr\,dt
  \]
  and use integration by parts to see that the right side is equivalent to
\begin{multline}\label{gwhardy}
\frac{2}{1-c} \int_0^\infty r^{-1} e^{\sigma_U(T-r)}
      (W(T,r))^2\,dr
    + \frac{2}{c-1} \int_0^\infty r^{-1} e^{\sigma_U(4-r)}
      (W(4,r))^2\,dr\\-\frac{2c}{c-1} \int_4^T\int_0^\infty
      r^{-2} e^{\sigma_U(t-r)} W^2\,dr\,dt
   +\frac{4}{c-1}\int_4^T\int_0^\infty r^{-1} e^{\sigma_U(t-r)} W(\partial_t+c\partial_r)W\,dr\,dt.
 \end{multline}
 Using the Cauchy-Schwarz inequality to bound
  \begin{multline} 
  \frac{4}{c-1}\int_4^T\int_0^\infty r^{-1} e^{\sigma_U(t-r)} W (\partial_t+c\partial_r)W\,dr\,dt
    \\\le \frac{c}{2(c-1)} \int_4^T\int_0^\infty r^{-2}
    e^{\sigma_U(t-r)} W^2\,dr\,dt
    + C \int_4^T\int_0^\infty e^{\sigma_U(t-r)} (\duuc
    W)^2\,dr\,dt,
  \end{multline}
plugging this into \eqref{gwhardy}, and omitting nonpositive terms from the right produces the desired result. 
\end{proof}

\medskip
\section{Proof of Theorem \ref{main_theorem}}


To solve \eqref{ohta_1d_system_odd}, we set up an iteration.
Let $W_0\equiv V_0\equiv 0$, and let $V_j, W_j$ solve
\begin{equation}
  \label{ohta_j}
  \begin{cases}
    \Box V_j = x^{-1} W_{j-1} \partial_t V_{j-1},\\
    \Box_c W_j =x^{-1} (\partial_t V_{j-1})^2,\\
    V_j(4,\cd)=\varepsilon V_{(0)},\quad \partial_t V_j(4,\cd)=\varepsilon
    V_{(1)},\\
    W_j(4,\cd)=\varepsilon W_{(0)},\quad \partial_t W_j(4,\cd)=\varepsilon W_{(1)}.
  \end{cases}
\end{equation}

We shall show that the sequence $((V_j, W_j))_{j=0}^\infty$ is Cauchy on
$[4,T_\varepsilon]\times \R$ in an
appropriate sense, and by standard results the limit is the desired
solution.

\subsection{Boundedness}  We first show a uniform boundedness of the
sequence $((V_j,W_j))$.  Let
\begin{multline}
\label{M'}
  M_j = \|r^{-\frac{1}{4}}\la r\ra^{-\frac{1}{4}} \la u\ra^{\frac{1}{2}} S^{\le 7} \du
  V_j\|_{\ell^\infty_R \ell^2_\tau L^2_tL^2_r(C^{1,R}_\tau)}  
  + \|\la u_c\ra^{-\frac{1}{2}}\la u\ra ^{\frac{1}{2}} S^{\le 7}
  \du V_j\|_{\ell^\infty_{U_c}\ell^2_\tau L^2_tL^2_r(C^{c,U_c}_\tau)}
  \\+\|r^{-\frac{1}{4}}\la r\ra^{-\frac{1}{4}} \la \uu\ra^{\frac{1}{2}} S^{\le 7} \duu
  V_j\|_{\ell^\infty_R \ell^2_\tau L^2_tL^2_r(C^{1,R}_\tau)}  
  + \|\la u_c\ra^{-\frac{1}{2}}\la\uu\ra^{\frac{1}{2}} S^{\le 7}
  \duu V_j\|_{\ell^\infty_{U_c}\ell^2_\tau L^2_tL^2_r(C^{c,U_c}_\tau)}
  \\ +
  \|\la u\ra^{-\frac{1}{2}}\la\uu\ra^{\frac{1}{2}} S^{\le 7}
  \duu V_j\|_{\ell^\infty_{U}\ell^2_\tau L^2_tL^2_r(C^{1,U}_\tau)}
  +\|r^{-\frac{1}{4}}\la r\ra^{-\frac{1}{4}} S^{\le 10} \partial
  V_j\|_{\ell^\infty_R \ell^2_\tau L^2_tL^2_r(C^{1,R}_\tau)}  
  \\+ \|\la
  u_c\ra^{-\frac{1}{2}} S^{\le 10}
  \partial V_j\|_{\ell^\infty_{U_c}\ell^2_\tau L^2_tL^2_r(C^{c,U_c}_\tau)}
  +\|\la
  u\ra^{-\frac{1}{2}}S^{\le 10}
  \duu V_j\|_{\ell^\infty_{U}\ell^2_\tau L^2_tL^2_r(C^{1,U}_\tau)}
  \\
  +\|r^{\frac{1}{2}}\la
  u\ra^{-\frac{1}{2}} S^{\le 10}
  \duuc W_j\|_{\ell^\infty_{U}\ell^2_\tau L^2_tL^2_r(C^{1,U}_\tau)}
  +\|r^{-\frac{1}{2}}\la
  u\ra^{-\frac{1}{2}} S^{\le 10}
  W_j\|_{\ell^\infty_{U}\ell^2_\tau L^2_tL^2_r(C^{1,U}_\tau)}
\\ +\|S^{\le 10} \duuc
  W_j\|_{\ell^2_R \ell^2_\tau L^2_tL^2_r(C^{1,R}_\tau)}  
  +\|r^{-1} S^{\le 10}  W_j\|_{\ell^2_R \ell^2_\tau L^2_tL^2_r(C^{c,R}_\tau)}.
\end{multline}
We will label these terms $(I)_j, (II)_j,\dots, (XII)_j$.

Since $\Box V_1\equiv\Box_c W_1\equiv0$, the smallness of the initial
data
along with \eqref{leghost_v}, \eqref{rp_w}, \eqref{hardy}, and \eqref{hardy_mixed}
imply that there exists $C_0$ so that
\[M_1 \le C_0 \varepsilon.\]
Moreover, $V_1$ is supported within $C^1$.
Under the assumption that
\begin{equation}\label{bdd_ih}
M_k\le 2C_0\varepsilon, \quad\text{ for } 0\le k\le
j-1,
\end{equation}
and that $\text{supp } V_{j-1}\subset C^1$,
we shall show that there exists a fixed constant $C$ such that
\begin{equation}\label{bdd_goal}M_j^2 \le (C_0 \varepsilon)^2 + C (2C_0\varepsilon)^2
  (\log(2+T_\varepsilon))^{\frac{1}{2}} M_j
\end{equation}
for $(t,r)\in [4,T_\varepsilon]\times \R_+$.  It will also follow
immediately that $\text{supp }V_j\subset C^1$.  
Using \eqref{lifespan} and absorbing the $M_j$ back into the left
side, this implies
\[M_j^2 \le 2(C_0 \varepsilon)^2 + \tilde{C} \tilde{c}
  \varepsilon^2.\]
Provided the $\tilde{c}$ in \eqref{lifespan} is chosen to be sufficiently small,
this shows that
\[M_j \le 2C_0\varepsilon\]
and completes the inductive proof of boundedness.

\subsubsection{Decay bounds}
To aid in the proof of \eqref{bdd_goal}, we shall first show some
auxiliary decay bounds provided \eqref{bdd_ih} holds.  In particular,
we start with proofs that, for $k\le j-1$:
\begin{align}
  \label{ptwise_W_crct}
  \|\la \uu_c \ra^{\frac{1}{2}} r^{-\frac{1}{2}}S^{\le 7} W_k\|_{\ell^2_\tau \ell^2_{R\le c\tau/2} L^\infty_t
    L^\infty_r(C^{c,R}_\tau)} &\lesssim \varepsilon,\\
  \label{ptwise_W_cuct}
   \|\la u_c\ra^{\frac{1}{2}}\la r\ra^{-\frac{1}{2}}
   S^{\le 7} W_k\|_{\ell^2_\tau \ell^2_{U_c\le c\tau/4}
  L^\infty_tL^\infty_r(C^{c,U_c}_\tau)}  &\lesssim \varepsilon,\\
    \label{ptwise_V_crt}
    \|\la\uu\ra S^{\le 5} \partial V_k\|_{\ell^\infty_R
  \ell^2_{\tau\ge 2R} L^\infty_t
    L^\infty_r(C^{1,R}_\tau)} &\lesssim \varepsilon,\\
  \label{ptwise_V_cut}
  \|\la u \ra S^{\le 5} \partial V_k\|_{\ell^
  \infty_{U} \ell^\infty_{\tau\ge 4U} L^\infty_t L^\infty_r(C^{1,U}_\tau)} &\lesssim \varepsilon.
\end{align}

We shall provide the proofs of these in order.  The first two are easy
corollaries of Corollary \ref{cor_wklsob}.
\begin{proof}[Proof of \eqref{ptwise_W_crct}]
This is an immediate consequence of \eqref{wcrt} and \eqref{bdd_ih}.
We apply \eqref{wcrt} to $S^{\le 7}W_k$ and note that if $W_k$ is odd then so
is $SW_k$.  Then,
indeed, 
\begin{align*} \|\la \uu_c \ra^{\frac{1}{2}} r^{-\frac{1}{2}}S^{\le 7} W_k&\|_{\ell^2_\tau \ell^2_{R\le c\tau/2} L^\infty_t
    L^\infty_r(C^{c,R}_\tau)} \\&\lesssim \|\la r\ra^{-1} S^{\le 9}
  W_k\|_{\ell^2_\tau \ell^2_{R\le c\tau/2} L^2_tL^2_r
    (\tC^{c,R}_\tau)}
  +\|S^{\le 8} \duuc W_k\|_{\ell^2_\tau\ell^2_{R\le c\tau/2}
                                L^2_tL^2_r(\tC^{c,R}_\tau)}\\
  &\lesssim (XII)_k+(XI)_k,
\end{align*}
and the bound follows from \eqref{bdd_ih}.
\end{proof}

\begin{proof}[Proof of \eqref{ptwise_W_cuct}]
Applying \eqref{wcut}, we have
\begin{align*}
  \|\la u_c\ra^{\frac{1}{2}}\la r\ra^{-\frac{1}{2}}
  S^{\le 7} W_k&\Bigr\|_{\ell^2_\tau \ell^2_{U_c\le c\tau/4}
  L^\infty_tL^\infty_r(C^{c,U_c}_\tau)} \\& \lesssim \|\la r\ra^{-1}
                                          S^{\le 9}
                                          W_k\|_{\ell^2_\tau\ell^2_{U_c\le
                                          c\tau/4}
                                          L^2_tL^2_r(\tC^{c,U_c}_\tau)}
                                          + \|S^{\le 8} \duuc
                                          W_k\|_{\ell^2_\tau
                                          \ell^2_{U_c\le c\tau/4} L^2_tL^2_r(\tC^{c,U_c}_\tau)}
  \\
  &\lesssim  (XII)_k+(XI)_k.
\end{align*}
And \eqref{ptwise_W_cuct} is then a direct consequence of \eqref{bdd_ih}.
\end{proof}

We now proceed to the pointwise bounds for $V_k$ with $k\le j-1$ using Lemma
\ref{vbounds}.  We first consider the bound away from the light cone.

\begin{proof}[Proof of \eqref{ptwise_V_crt}]
    Applying \eqref{kscrt}, and commuting $S$ with $\Box$ and $\partial$, we have
    \begin{multline*}
        \|\la\uu\ra S^{\le 5} \partial V_k \|_{\ell^\infty_R
          \ell^2_{\tau \ge 2R}L^\infty_t
        L^\infty_r(C^{1,R}_\tau)}
        \\\lesssim 
        \|\la r \ra^{-\frac{1}{2}} \la\uu\ra^\frac{1}{2} S^{\le 7} \partial V_k\|_{\ell^\infty_R
          \ell^2_{\tau \ge 2R}L^2_t
        L^2_r(\tC^{1,R}_\tau)}
        + \|\la r \ra^\frac{1}{4} r^{\frac{1}{4}}\la\uu\ra^\frac{1}{2} S^{\le 6} \Box V_k\|_{\ell^\infty_R
          \ell^2_{\tau \ge 2R}L^2_t L^2_r (\tC^{1,R}_\tau)}.
  \end{multline*}
  Splitting $\partial$ into $(\du, \duu)$, we bound the first term by 
   \[
        \|\la r \ra^{-\frac{1}{2}} \la \uu\ra^\frac{1}{2} S^{\le 7}
      \duu V_k\|_{ \ell^\infty_R
          \ell^2_{\tau \ge 2R}L^2_t L^2_r
                (\tC^{1,R}_\tau)} 
        + \|\la r \ra^{-\frac{1}{2}} \la u\ra^\frac{1}{2} S^{\le 7}
                \du V_k\|_{ \ell^\infty_R
          \ell^2_{\tau \ge 2R}L^2_t
                L^2_r (\tC^{1,R}_\tau)} 
        \lc
        (III)_k+(I)_k,\]
    where we exploit that $u\approx\uu$ on $\tC^{1, R}_\tau$ regions with
    $R\le \tau/2$.  This term is $\O(\varepsilon)$ by \eqref{bdd_ih}.

    Moving to the second term, we crudely apply the product rule to $S^{\le6}\Box V_k$, giving
    \begin{equation}\label{prodrule}|S^{\le 6} \Box V_k|\le r^{-1} |S^{\le 6} W_{k-1}||S^{\le 6}
      \partial_t V_{k-1}|,
    \end{equation}
    where we note that $|Sr^{-1}|=r^{-1}$.  On $\tC^{1,R}_\tau$ with
    $R\le \tau/2$, we have
    \[\la r\ra \lesssim \la \uu_c\ra \approx \la \uu\ra \approx \la
      u\ra\approx \tau.\]
    This allows us to bound:
\begin{multline*}   
 \|\la r \ra^\frac{1}{4} r^{\frac{1}{4}}\la\uu\ra^\frac{1}{2} S^{\le
   6} \Box V_k\|_{\ell^\infty_R \ell^2_{\tau \ge 2R} L^2_t L^2_r
   (\tC^{1,R}_\tau)}
 \lesssim \|r^{-\frac{1}{2}}\la \uu_c\ra^{\frac{1}{2}}
 S^{\le 6}
 W_{k-1}\|_{\ell^\infty_\tau\ell^\infty_{R\le\tau/2}L^\infty_tL^\infty_r(\tC^{1,R}_\tau)}
\\\times
 \Bigl(
   \|r^{-\frac{1}{4}}\la r\ra^{-\frac{1}{4}} \la
   \uu\ra^{\frac{1}{2}}S^{\le 6} \partial_\uu V_{k-1}\|_{\ell^\infty_R \ell^2_{\tau\ge 2R}L^2_tL^2_r(\tC^{1,R}_\tau)} +
   \|r^{-\frac{1}{4}}\la r\ra^{-\frac{1}{4}} \la
   u\ra^{\frac{1}{2}}S^{\le 6}\partial_u V_{k-1}\|_{\ell^\infty_R \ell^2_{\tau\ge 2R}L^2_tL^2_r(\tC^{1,R}_\tau)} \Bigr).
 \end{multline*}
We note that since $c>1$, each $\tC^{1,R}_\tau$ with $R\le \tau/2$ is contained in a
finite number of regions $C^{c,\tR}_\ttau$ with $\tR\le c\ttau/2$.
This allows us to apply \eqref{ptwise_W_crct} to the first factor in the
right side, which shows that this is $\lesssim \varepsilon
[(III)_{k-1}+(I)_{k-1}]$.  The desired bound is then a consequence of \eqref{bdd_ih}.\end{proof}

We finally consider the pointwise bound for $\partial V_k$ away from the
light cone.
\begin{proof}[Proof of \eqref{ptwise_V_cut}]  Fixing $U$ and $\tau$
  with $U\le \tau/4$, 
we start with an application of \eqref{kscut}, which gives that
\begin{equation}\label{ptwise_V_cut_1}
  \|\la u\ra S^{\le 5}\partial
  V_k\|_{L^\infty_tL^\infty_r(C^{1,U}_\tau)}
  \lesssim \|\la u\ra^{\frac{1}{2}}\la \uu\ra^{-\frac{1}{2}}
  S^{\le 7} \partial V_k\|_{L^2_tL^2_r(\tC^{1,U}_\tau)}
  + \|\la u\ra^{\frac{1}{2}}\la \uu\ra^{\frac{1}{2}} S^{\le 6} \Box
  V_k\|_{L^2_tL^2_r(\tC^{1,U}_\tau)}.
\end{equation}
For the first term, we have
\begin{align*}\|\la u\ra^{\frac{1}{2}}\la \uu\ra^{-\frac{1}{2}}
  S^{\le 7} \partial V_k\|_{L^2_tL^2_r(\tC^{1,U}_\tau)}
&\lesssim 
\|\la r\ra^{-\frac{1}{2}} \la u\ra^{\frac{1}{2}}
  S^{\le 7} \partial_u V_k\|_{L^2_tL^2_r(\tC^{1,U}_\tau)}
  +
  \|\la r\ra^{-\frac{1}{2}}\la \uu\ra^{\frac{1}{2}}
                                                          S^{\le 7} \partial_\uu V_k\|_{L^2_tL^2_r(\tC^{1,U}_\tau)}\\
  &\le (I)_k + (III)_k,
\end{align*}
and the bound is an immediate consequence of \eqref{bdd_ih}.

For the second term in \eqref{ptwise_V_cut_1}, we use
\eqref{prodrule}.
In order to use \eqref{ptwise_W_crct} and \eqref{ptwise_W_cuct}, we
will consider separately when $\tC^{1,U}_\tau$ intersects $\bigcup_{R\le
  c\ttau/2} C^{c,R}_\ttau$ and $\bigcup_{U_c\le c\ttau/4} C^{c,U_c}_\ttau$.
 Away
from the speed $c$ light cone, we have
\begin{multline*}
  \|\la u\ra^{\frac{1}{2}}\la \uu\ra^{\frac{1}{2}} S^{\le 6} \Box
  V_k\|_{L^2_tL^2_r(\tC^{1,U}_\tau \cap (\bigcup_{\ttau}\bigcup_{R\le c\ttau/2}
    C^{c,R}_\ttau))}
  \lesssim \|\la \uu_c\ra^{\frac{1}{2}} r^{-\frac{1}{2}} S^{\le 6}
  W_{k-1} \|_{\ell^\infty_{\ttau} \ell^\infty_{R\le c\ttau/2}
    L^\infty_tL^\infty_r(C^{c,R}_\ttau)}
  \\\times
 \Bigl(\|\la r\ra^{-\frac{1}{2}} \la u\ra^{\frac{1}{2}}S^{\le
   6}\partial_u V_{k-1}\|_{L^2_tL^2_r(\tC^{1,U}_\tau)}
   +\|\la r\ra^{-\frac{1}{2}} \la \uu\ra^{\frac{1}{2}}S^{\le
   6}\partial_\uu V_{k-1}\|_{L^2_tL^2_r(\tC^{1,U}_\tau)}
   \Bigr),
 \end{multline*}
as $\tC^{1,U}_\tau$ only intersects $\bigcup_{R\le c\ttau/2}
C^{c,R}_\ttau$ for a finite number of $\ttau$.
 Since $\tC^{1,U}_\tau$ (for fixed $U$ and $\tau$) is contained in a finite number of dyadic
regions $C^{1,R}_\ttau$, it follows, using \eqref{ptwise_W_crct} that this is
\[\lesssim \varepsilon [(I)_{k-1}+(III)_{k-1}].\]
After applying \eqref{bdd_ih}, supremums can be taken over $U$ and
$\tau$ to obtain the desired result.

Near the speed $c$ light cone, we instead see
\begin{multline*}
  \|\la u\ra^{\frac{1}{2}}\la \uu\ra^{\frac{1}{2}} S^{\le 6} \Box
  V_k\|_{L^2_tL^2_r(\tC^{1,U}_\tau \cap (\bigcup_{\ttau}\bigcup_{U_c\le c\ttau/4}
    C^{c,U_c}_\ttau))}
  \lesssim \|\la u_c\ra^{\frac{1}{2}} r^{-\frac{1}{2}} S^{\le 6}
  W_{k-1} \|_{\ell^2_{\ttau}\ell^2_{U_c\le c\ttau/4}
    L^\infty_tL^\infty_r(C^{c,U_c}_\ttau)}
  \\\times
 \Bigl(\|\la u_c\ra^{-\frac{1}{2}} \la u\ra^{\frac{1}{2}}S^{\le
   6}\partial_u V_{k-1}\|_{\ell^\infty_{\ttau}\ell^\infty_{U_c} L^2_tL^2_r(C^{c,U_c}_\ttau)}
   +\|\la u_c\ra^{-\frac{1}{2}} \la \uu\ra^{\frac{1}{2}}S^{\le
   6}\partial_\uu V_{k-1}\|_{\ell^\infty_{\ttau}\ell^\infty_{U_c}L^2_tL^2_r(\tC^{c,U_c}_\ttau)}
   \Bigr).
\end{multline*}
By \eqref{ptwise_W_cuct}, this is bounded by
\[\varepsilon [(II)_{k-1}+(IV)_{k-1}].\]
Thus, an application of \eqref{bdd_ih} and taking supremums over $U$,
$\tau$ yields the result.
\end{proof}

\subsubsection{Bound on terms $(I)_j,\dots, (V)_j$}
By applying \eqref{leghost_v} (with $p=1$) to $S^{\le 7} V_j$, we
see that
\begin{multline}\label{ItoV}
 (I)_j^2 +\dots + (V)_j^2 \le (C_0\varepsilon)^2 +
C  \int_4^{T_\varepsilon}\int_0^\infty \la\uu\ra |S^{\le 7} \Box V_j||\duu S^{\le 7}
  V_j|\,dr\,dt
\\  +C\int_4^{T_\varepsilon}\int_0^\infty \la u\ra |S^{\le 7} \Box V_j| |\du S^{\le 7} V_j|\,dr\,dt.
\end{multline}
We need to show that the latter two terms are bounded by
\begin{equation}\label{ItoVgoal}C\varepsilon^2
  (\log(2+T_\varepsilon))^{\frac{1}{2}}M_j.
\end{equation}
Due to finite speed of propagation, we note $\partial_t V_{j-1}$, and
hence $\Box V_j$, vanishes for $r \ge t-3$.

We shall first decompose these integrals using \eqref{decomp_c} at
speed $c$ and note that $C^1\subset C^c$.  For the
first of the integrals in \eqref{ItoV}, we have
\begin{multline}\label{uucdecomp}\|\la \uu\ra S^{\le 7}\Box V_j\cdot S^{\le 7} \partial_\uu
  V_j\|_{L^1_tL^1_r(C^1)}
  \lesssim \|\la \uu\ra S^{\le 7}\Box V_j\cdot S^{\le 7} \partial_\uu
  V_j\|_{\ell^1_\tau L^1_tL^1_r(C^{c,R=1}_\tau)}
  \\+\|\la \uu\ra S^{\le 7}\Box V_j\cdot S^{\le 7} \partial_\uu
  V_j\|_{\ell^1_\tau \ell^1_{1<R\le c\tau/2} L^1_tL^1_r(C^{c,R}_\tau)}
 +\|\la \uu\ra S^{\le 7}\Box V_j\cdot S^{\le 7} \partial_\uu
  V_j\|_{\ell^1_\tau\ell^1_{U_c\le c\tau/4} L^1_tL^1_r(C^{c,U_c}_\tau)}.
\end{multline}
A naive application of the product rule gives that
 \[|S^{\le 7} \Box V_j|\le r^{-1} |S^{\le 7} W_{j-1}| \Bigl(|S^{\le 7}
   \duu V_{j-1}|+|S^{\le 7} \du V_{j-1}|\Bigr),\]
which we will apply in each instance.

For the first term in the right side of \eqref{uucdecomp}, since $\la
\uu\ra\approx \la u\ra$ on $C^{c,R=1}_\tau$, we have
 \begin{multline*}
   \|\la \uu\ra S^{\le 7}\Box V_j \cdot S^{\le 7} \partial_{\uu}
   V_j\|_{\ell^1_\tau L^1_tL^1_r(C^{c,R=1}_\tau)}
 \lesssim \|\la \uu_c\ra^{\frac{1}{2}} r^{-\frac{1}{2}} S^{\le 7}
   W_{j-1}\|_{\ell^2_\tau L^\infty_tL^\infty_r(C^{c,R=1}_\tau)}\\\times
 \Bigl(\|\la r\ra^{-\frac{1}{4}}r^{-\frac{1}{4}}  \la \uu\ra^{\frac{1}{2}} S^{\le 7}\partial_\uu V_{j-1}
   \|_{\ell^\infty_\tau L^2_t L^2_r(C^{c,R=1}_\tau)}+  \|\la r\ra^{-\frac{1}{4}}r^{-\frac{1}{4}}  \la u\ra^{\frac{1}{2}} S^{\le 7}\partial_u V_{j-1}
   \|_{\ell^\infty_\tau L^2_t L^2_r(C^{c,R=1}_\tau)}\Bigr)  \\\times \|\la r\ra^{-\frac{1}{4}}r^{-\frac{1}{4}} \la \uu\ra^{\frac{1}{2}}
   S^{\le 7}\partial_{\uu} V_j\|_{\ell^2_\tau L^2_tL^2_r(C^{c,R=1}_\tau)},
 \end{multline*}
 which, using \eqref{ptwise_W_crct}, is bounded by $C\varepsilon [(III)_{j-1}+(I)_{j-1}] (III)_j.$
 Since we have \eqref{bdd_ih}, this term is controlled by
 \eqref{ItoVgoal} as desired.

For the next term in \eqref{uucdecomp}, we use the Schwarz inequality
and the facts that $\la \uu\ra \approx \tau$ and $\la r\ra\lesssim
\tau$ on $C^{c, R}_\tau$ with $R\le c\tau/2$
to see that
 \begin{multline}\label{lowVR}
   \|\la \uu\ra S^{\le 7}\Box V_j \cdot S^{\le 7} \partial_{\uu}
   V_j\|_{\ell^1_\tau \ell^1_{1<R\le c\tau/2}  L^1_tL^1_r(C^{c,R}_\tau)}
 \\\lesssim \|\la \uu_c\ra^{\frac{1}{2}} r^{-\frac{1}{2}} S^{\le 7}
   W_{j-1}\|_{\ell^2_\tau \ell^2_{1<R\le c\tau/2} L^\infty_tL^\infty_r(C^{c,R}_\tau)}
 \|\la r\ra^{-\frac{1}{4}}r^{-\frac{1}{4}}  \la \uu\ra^{\frac{1}{2}} S^{\le 7}\partial_\uu V_{j-1}
 \|_{\ell^\infty_\tau \ell^\infty_R L^2_t L^2_r(C^{c,R}_\tau)} \\\times\|\la r\ra^{-\frac{1}{4}}r^{-\frac{1}{4}} \la \uu\ra^{\frac{1}{2}}
  S^{\le 7} \partial_{\uu} V_j\|_{\ell^2_\tau \ell^2_{1<R\le c\tau/2}  L^2_tL^2_r(C^{c,R}_\tau)}
\\+
\|\la \uu_c\ra^{\frac{1}{2}} r^{-\frac{1}{2}} S^{\le 7}
   W_{j-1}\|_{\ell^2_\tau \ell^2_{1<R\le c\tau/2} L^\infty_tL^\infty_r(C^{c,R}_\tau)}
 \|\la r\ra^{-\frac{1}{2}}  \la u\ra^{\frac{1}{2}} S^{\le 7}\partial_u V_{j-1}
   \|_{\ell^\infty_\tau \ell^\infty_R L^2_t L^2_r(C^{c,R}_\tau)}  \\\times \|\la u\ra^{-\frac{1}{2}} \la \uu\ra^{\frac{1}{2}}
   S^{\le 7} \partial_{\uu} V_j\|_{\ell^2_\tau \ell^2_{1<R\le c\tau/2}  L^2_tL^2_r(C^{c,R}_\tau)}.
 \end{multline}
For the last factor in each term, we sum back up and re-decompose
in terms of speed $1$ regions to see
\begin{equation}\label{changespeedR}\begin{split}\|r^{-\frac{1}{4}}\la r\ra^{-\frac{1}{4}} \la \uu\ra^{\frac{1}{2}}
   S^{\le 7}\partial_{\uu} V_j\|_{\ell^2_\tau \ell^2_{1<R\le c\tau/2}
     L^2_tL^2_r(C^{c,R}_\tau)} &+ \|r^{-\frac{1}{4}}\la r\ra^{-\frac{1}{4}} \la \uu\ra^{\frac{1}{2}}
  S^{\le 7} \partial_{\uu} V_j\|_{\ell^2_\tau \ell^2_{U_c\le c\tau/4}
     L^2_tL^2_r(C^{c,U_c}_\tau)} \\
&\lesssim \|r^{-\frac{1}{4}}\la r\ra^{-\frac{1}{r}} \la \uu\ra^{\frac{1}{2}}
  S^{\le 7} \partial_{\uu} V_j\|_{L^2_tL^2_r(C^1)}\\
&\lesssim  (\log(2+T_\varepsilon))^{\frac{1}{2}} \|r^{-\frac{1}{4}}\la r\ra^{-\frac{1}{4}} \la \uu\ra^{\frac{1}{2}}
   S^{\le 7}\partial_{\uu} V_j\|_{\ell^\infty_R \ell^2_{\tau}
     L^2_tL^2_r(C^{1,R}_\tau)}
 \end{split}
\end{equation}
and similarly
\begin{multline}\label{changespeedU}\|\la u\ra^{-\frac{1}{2}} \la \uu\ra^{\frac{1}{2}}
  S^{\le 7} \partial_{\uu} V_j\|_{\ell^2_\tau \ell^2_{1<R\le c\tau/2}
     L^2_tL^2_r(C^{c,R}_\tau)} + \|\la u\ra^{-\frac{1}{2}} \la \uu\ra^{\frac{1}{2}}
  S^{\le 7} \partial_{\uu} V_j\|_{\ell^2_\tau \ell^2_{U_c\le c\tau/4}
     L^2_tL^2_r(C^{c,U_c}_\tau)}
\\\lesssim  (\log(2+T_\varepsilon))^{\frac{1}{2}} \|\la u\ra^{-\frac{1}{2}} \la \uu\ra^{\frac{1}{2}}
   S^{\le 7}\partial_{\uu} V_j\|_{\ell^\infty_U \ell^2_{\tau}
     L^2_tL^2_r(C^{1,U}_\tau)}.
\end{multline}
Applying \eqref{ptwise_W_crct}, it then follows that the left side of
\eqref{lowVR} is bounded by
\[C\varepsilon (III)_{j-1} (\log(2+T_\varepsilon))^{\frac{1}{2}}
  (III)_j
+ C\varepsilon (I)_{j-1} (\log(2+T_\varepsilon))^{\frac{1}{2}} (V)_j,
\]
which, owing to \eqref{bdd_ih}, is controlled by \eqref{ItoVgoal}.

The last term in \eqref{uucdecomp} is handled similarly, but we now
must rely upon \eqref{ptwise_W_cuct}.  Here we use the fact that
$r\approx \tau \approx \la \uu\ra$ on $C^{c,U_c}_\tau$ with $1\le U_c\le c\tau/4$.
Indeed, 
 \begin{multline}\label{lowVU}
   \|\la \uu\ra S^{\le 7}\Box V_j \cdot S^{\le 7} \partial_{\uu}
   V_j\|_{\ell^1_\tau \ell^1_{U_c\le c\tau/4}  L^1_tL^1_r(C^{c,U_c}_\tau)}
 \\\lesssim \|\la u_c\ra^{\frac{1}{2}} r^{-\frac{1}{2}} S^{\le 7}
   W_{j-1}\|_{\ell^2_\tau \ell^2_{U_c\le c\tau/4} L^\infty_tL^\infty_r(C^{c,U_c}_\tau)}
 \|  \la u_c\ra^{-\frac{1}{2}} \la \uu\ra^{\frac{1}{2}} S^{\le 7}\partial_\uu V_{j-1}
 \|_{\ell^\infty_\tau \ell^\infty_{U_c} L^2_t L^2_r(C^{c,U_c}_\tau)}
 \\\times\|\la r\ra^{-\frac{1}{4}}r^{-\frac{1}{4}}\la \uu\ra^{\frac{1}{2}}
 S^{\le 7}  \partial_{\uu} V_j\|_{\ell^2_\tau \ell^2_{U_c\le c\tau/4}  L^2_tL^2_r(C^{c,U_c}_\tau)}
\\+
\|\la u_c\ra^{\frac{1}{2}} r^{-\frac{1}{2}} S^{\le 7}
   W_{j-1}\|_{\ell^2_\tau \ell^2_{U_c\le c\tau/4} L^\infty_tL^\infty_r(C^{c,U_c}_\tau)}
 \|\la u_c\ra^{-\frac{1}{2}}  \la u\ra^{\frac{1}{2}} S^{\le 7}\partial_u V_{j-1}
   \|_{\ell^\infty_\tau \ell^\infty_{U_c} L^2_t L^2_r(C^{c,U_c}_\tau)}  \\\times \|\la u\ra^{-\frac{1}{2}} \la \uu\ra^{\frac{1}{2}}S^{\le 7}
   \partial_{\uu} V_j\|_{\ell^2_\tau \ell^2_{U_c\le c\tau/4}  L^2_tL^2_r(C^{c,U_c}_\tau)}.
 \end{multline}
Using \eqref{ptwise_W_cuct}, \eqref{changespeedR}, and \eqref{changespeedU}, it follows that
this is
\[\lesssim \varepsilon (IV)_{j-1}
  (\log(2+T_\varepsilon))^{\frac{1}{2}} (III)_j + \varepsilon
  (IV)_{j-1} (\log(2+T_\varepsilon))^{\frac{1}{2}} (V)_j.\]
The inductive hypothesis \eqref{bdd_ih} then gives that this is
bounded by \eqref{ItoVgoal} as desired.

The same strategy bounds the second integral in \eqref{ItoV}.  In
fact, since $\la u\ra\lesssim \la \uu\ra$ in all regions, the argument
can be simplified.  Indeed, we have
 \begin{multline}\label{lowVRu}
   \|\la u\ra S^{\le 7}\Box V_j \cdot S^{\le 7} \partial_{u}
   V_j\|_{\ell^1_\tau \ell^1_{R\le c\tau/2}  L^1_tL^1_r(C^{c,R}_\tau)}
 \lesssim \|\la \tau\ra^{\frac{1}{2}} r^{-\frac{1}{2}} S^{\le 7}
   W_{j-1}\|_{\ell^2_\tau \ell^2_{R\le c\tau/2} L^\infty_tL^\infty_r(C^{c,R}_\tau)}
\\\times\Bigl( \|\la r\ra^{-\frac{1}{4}}r^{-\frac{1}{4}}  \la \uu\ra^{\frac{1}{2}} S^{\le 7}\partial_\uu V_{j-1}
\|_{\ell^\infty_\tau \ell^\infty_R L^2_t L^2_r(C^{c,R}_\tau)}
+\|\la r\ra^{-\frac{1}{4}}r^{-\frac{1}{4}}  \la u\ra^{\frac{1}{2}} S^{\le 7}\partial_u V_{j-1}
\|_{\ell^\infty_\tau \ell^\infty_R L^2_t L^2_r(C^{c,R}_\tau)}
\Bigr)
\\\times\|\la r\ra^{-\frac{1}{4}}r^{-\frac{1}{4}} \la u\ra^{\frac{1}{2}}
   S^{\le 7}\partial_{u} V_j\|_{\ell^2_\tau \ell^2_{R\le c\tau/2}  L^2_tL^2_r(C^{c,R}_\tau)}.
 \end{multline}
 Using a direct analog of \eqref{changespeedR} and \eqref{ptwise_W_crct}, this is
 \[\lesssim \varepsilon ((III)_{j-1}+(I)_{j-1})
   (\log(2+T_\varepsilon))^{\frac{1}{2}} (I)_j,\]
 which is in turn controlled by \eqref{ItoVgoal}.

 And
  \begin{multline}\label{lowVUu}
   \|\la u\ra S^{\le 7}\Box V_j \cdot S^{\le 7} \partial_{u}
   V_j\|_{\ell^1_\tau \ell^1_{U_c\le c\tau/4}  L^1_tL^1_r(C^{c,U_c}_\tau)}
 \lesssim \|\la u_c\ra^{\frac{1}{2}} r^{-\frac{1}{2}} S^{\le 7}
   W_{j-1}\|_{\ell^2_\tau \ell^2_{U_c\le c\tau/4} L^\infty_tL^\infty_r(C^{c,U_c}_\tau)}
 \\\times \Bigl(\|  \la u_c\ra^{-\frac{1}{2}} \la \uu\ra^{\frac{1}{2}} S^{\le 7}\partial_\uu V_{j-1}
 \|_{\ell^\infty_\tau \ell^\infty_{U_c} L^2_t L^2_r(C^{c,U_c}_\tau)} + \|  \la u_c\ra^{-\frac{1}{2}} \la u\ra^{\frac{1}{2}} S^{\le 7}\partial_u V_{j-1}
 \|_{\ell^\infty_\tau \ell^\infty_{U_c} L^2_t L^2_r(C^{c,U_c}_\tau)}\Bigr)
 \\\times\|\la r\ra^{-\frac{1}{4}}r^{-\frac{1}{4}}\la u\ra^{\frac{1}{2}}
  S^{\le 7} \partial_{u} V_j\|_{\ell^2_\tau \ell^2_{U_c\le c\tau/4}  L^2_tL^2_r(C^{c,U_c}_\tau)}.
 \end{multline}
Using the analog of \eqref{changespeedU} and \eqref{ptwise_W_cuct},
this is
\[\lesssim \varepsilon ((IV)_{j-1}+(II)_{j-1})
  (\log(2+T_\varepsilon))^{\frac{1}{2}} (I)_j,\]
which is in turn controlled by \eqref{ItoVgoal} after applying
\eqref{bdd_ih}.  This completes the proof of the boundedness of terms
$(I)_j,\dots, (V)_j$.

\subsubsection{Bound on terms $(VI)_j,\dots, (VIII)_j$}  We now proceed
to considering the high order energy bounds for $V_j$.  We begin by applying \eqref{leghost_v} (with $p=0$) to $S^{\le 10} V_j$ to obtain
\[
(VI)^2_j+(VII)^2_j+(VIII)^2_j \le (C_0\varepsilon)^2 +C\int_4^{T_\varepsilon}\int |\Box S^{\le 10} V_j| |\partial S^{\le 10} V_j| \,dr\,dt. 
\]
Again, we need to show that the last term is bounded by \eqref{ItoVgoal}.
This time, however, we need to apply the product rule more
carefully. We note that on each term, there will be one factor with no
more than half of the 10 total vector fields. As such, we have 
\[|S^{\le 10} \Box V_j| \le r^{-1}|S^{\le 5} W_{j-1}||S^{\le
    10}\partial V_{j-1}| +r^{-1}|S^{\le 10} W_{j-1}||S^{\le 5}\partial
  V_{j-1}|.\]

When $W$ is lower order, we shall initially decompose in speed $c$
regions.  On $C^{c,R}_\tau$ regions, since $\la r\ra\lesssim \tau$, we bound
\begin{multline}\label{lowWR}
   \|r^{-1}S^{\le 5} W_{j-1}S^{\le 10}\partial V_{j-1} S^{\le 10}
   \partial V_j\|_{\ell^1_\tau \ell^1_{R\le c\tau/2} L^1_t L^1_r
     (C^{c,R}_\tau)}
   \lesssim \|\la \uu_c\ra^{\frac{1}{2}} r^{-\frac{1}{2}} S^{\le 5}
   W_{j-1}\|_{\ell^2_\tau \ell^2_{R\le c\tau/2}L^\infty_tL^\infty_r(C^{c,R}_\tau)}
   \\\times  \|r^{-\frac{1}{4}} \la r\ra^{-\frac{1}{4}} S^{\le 10} \partial
       V_{j-1}\|_{\ell^\infty_\tau \ell^\infty_R
         L^2_tL^2_r(C^{c,R}_\tau)}
       \|r^{-\frac{1}{4}} \la r\ra^{-\frac{1}{4}} S^{\le 10} \partial
       V_j\|_{\ell^2_\tau \ell^2_{R\le c\tau/2} L^2_tL^2_r(C^{c,R}_\tau)}
\end{multline}
Arguing as in \eqref{changespeedR}, the last factor satisfies
\begin{multline}\label{highchangespeed} \|r^{-\frac{1}{4}} \la r\ra^{-\frac{1}{4}} S^{\le 10} \partial
       V_j\|_{\ell^2_\tau \ell^2_{R\le c\tau/2}
         L^2_tL^2_r(C^{c,R}_\tau)}
+       \|r^{-\frac{1}{4}} \la r\ra^{-\frac{1}{4}} S^{\le 10} \partial
       V_j\|_{\ell^2_\tau \ell^2_{U_c\le c\tau/4}
         L^2_tL^2_r(C^{c,U_c}_\tau)}
       \\\lesssim (\log(2+T_\varepsilon))^{\frac{1}{2}}  \|r^{-\frac{1}{4}} \la r\ra^{-\frac{1}{4}} S^{\le 10} \partial
       V_j\|_{\ell^\infty_R \ell^2_\tau L^2_tL^2_r(C^{1,R}_\tau)}.
     \end{multline}
     Using this and \eqref{ptwise_W_crct}, we see that the left side of
\eqref{lowWR} is
\[\lesssim \varepsilon (VI)_{j-1}
  (\log(2+T_\varepsilon))^{\frac{1}{2}} (VI)_j,\]
and thus due to \eqref{bdd_ih} is controlled by \eqref{ItoVgoal}.
On the $C^{c,U_c}_\tau$ regions, we instead have 
\begin{multline}\label{lowWUc}
   \|r^{-1}S^{\le 5} W_{j-1}S^{\le 10}\partial V_{j-1} S^{\le 10}
   \partial V_j\|_{\ell^1_\tau \ell^1_{U_c\le c\tau/4} L^1_t L^1_r
     (C^{c,U_c}_\tau)}
   \lesssim \|\la u_c\ra^{\frac{1}{2}} \la r\ra^{-\frac{1}{2}} S^{\le 5}
   W_{j-1}\|_{\ell^2_\tau \ell^2_{U_c\le c\tau/4}L^\infty_tL^\infty_r(C^{c,U_c}_\tau)}
   \\\times  \|\la u_c\ra^{-\frac{1}{2}} S^{\le 10} \partial
       V_{j-1}\|_{\ell^\infty_\tau \ell^\infty_{U_c}
         L^2_tL^2_r(C^{c,U_c}_\tau)}
       \|r^{-\frac{1}{4}} \la r\ra^{-\frac{1}{4}} S^{\le 10} \partial
       V_j\|_{\ell^2_\tau \ell^2_{U_c\le c\tau/4} L^2_tL^2_r(C^{c,U_c}_\tau)}.
\end{multline}
By \eqref{ptwise_W_cuct} and \eqref{highchangespeed}, this is
\[\lesssim \varepsilon (VII)_{j-1}
  (\log(2+T_\varepsilon))^{\frac{1}{2}} (VI)_j,\]
which is in turn bounded by \eqref{ItoVgoal} upon applying \eqref{bdd_ih}.

When $\partial V$ is lower order, we instead use the speed 1
decomposition.  Away from the light cone, since $\la r\ra\le \tau$ on
$C^{1,R}_\tau$ with $R\le \tau/2$, we obtain
\begin{multline*}
    \|r^{-1}S^{\le 5} \partial V_{j-1}S^{\le 10}W_{j-1} S^{\le 10}
    \partial V_j|\|_{\ell^1_R \ell^1_{\tau\ge 2R} L^1_t L^1_r (C^{1,R}_\tau)}
  \\\lesssim \|\la \uu\ra S^{\le 5} \partial V_{j-1}\|_{\ell^\infty_R
    \ell^2_{\tau\ge 2R}L^\infty_tL^\infty_r(C^{1,R}_\tau)} \|r^{-1} S^{\le 10}
    W_{j-1}\|_{\ell^2_R\ell^2_\tau L^2_tL^2_r(C^{1,R}_\tau)}
    \\\times \|\la
    r\ra^{-\frac{1}{2}} S^{\le 10} \partial
    V_j\|_{\ell^\infty_R\ell^\infty_\tau L^2_tL^2_r(C^{1,R}_\tau)}.
\end{multline*}
Here we have used the additional power $\la \uu\ra^{-\frac{1}{2}}\le
\la r\ra^{-\frac{1}{2}}$ to control the remaining dyadic sum over
$R$.  We may apply \eqref{ptwise_V_crt} to then see that this is
$\lesssim \varepsilon (XI)_{j-1} (VI)_{j}$, which by \eqref{bdd_ih} is
better than the required bound \eqref{ItoVgoal}.

When $\partial V$ is lower order and we are near the speed $1$ light
cone, we instead have
\begin{multline*}
    \|r^{-1}S^{\le 5} \partial V_{j-1}S^{\le 10}W_{j-1} S^{\le 10}
    \partial V_j|\|_{\ell^1_U \ell^1_{\tau \ge 4U} L^1_t L^1_r (C^{1,U}_\tau)}
  \\\lesssim \|\la u\ra S^{\le 5} \partial V_{j-1}\|_{\ell^\infty_\tau
    \ell^\infty_{U\ge \tau/4}L^\infty_tL^\infty_r(C^{1,U}_\tau)} \|\la
  u\ra^{-\frac{1}{2}}r^{-\frac{1}{2}} S^{\le 10}
    W_{j-1}\|_{\ell^\infty_U \ell^2_\tau L^2_tL^2_r(C^{1,U}_\tau)}
    \\\times \|\la
    r\ra^{-\frac{1}{2}} S^{\le 10} \partial
    V_j\|_{\ell^2_\tau \ell^2_{U\le \tau/4}  L^2_tL^2_r(C^{1,U}_\tau)}.
\end{multline*}
Here we have again gained $\ell^2_U$ summability from the extra factor
of $\la u\ra^{-\frac{1}{2}}$.  We note that the last factor satisfies
\begin{equation}\label{UtoR}\begin{split}\|\la
    r\ra^{-\frac{1}{2}} S^{\le 10} \partial
    V_j\|_{\ell^2_\tau \ell^2_{U\le \tau/4}  L^2_tL^2_r(C^{1,U}_\tau)}
    &\lesssim \|\la
    r\ra^{-\frac{1}{2}} S^{\le 10} \partial
      V_j\|_{\ell^2_R \ell^2_{\tau}  L^2_tL^2_r(C^{1,R}_\tau)}\\
  &\lesssim (\log(2+T_\varepsilon))^{\frac{1}{2}}\|\la
    r\ra^{-\frac{1}{2}} S^{\le 10} \partial
      V_j\|_{\ell^\infty_R \ell^2_{\tau}  L^2_tL^2_r(C^{1,R}_\tau)}.
    \end{split}
  \end{equation}
  It follows from \eqref{ptwise_V_cut} that
\[\|r^{-1}S^{\le 5} \partial V_{j-1}S^{\le 10}W_{j-1} S^{\le 10}
    \partial V_j|\|_{\ell^1_U \ell^1_{\tau \ge 4U} L^1_t L^1_r
      (C^{1,U}_\tau)}
    \lesssim \varepsilon (X)_{j-1}
    (\log(2+T_\varepsilon))^{\frac{1}{2}} (VI)_j.\]
  This gives the bound by \eqref{ItoVgoal} upon using \eqref{bdd_ih}.

  \subsubsection{Bound on terms $(IX)_j,\dots, (XII)_j$}
Applying \eqref{rp_w}, \eqref{hardy}, and \eqref{hardy_mixed} to
$S^{\le 10} W_j$ gives
\[
(IX)^2_j+\dots+(XII)^2_j \le (C_0\varepsilon)^2
+C\int_4^{T_\varepsilon}\int r|\Box_c S^{\le 10} W_j||\duuc S^{\le 10} W_j|\,dr\,dt.
\]
We again seek to show that the last term is bounded by \eqref{ItoVgoal}.
The decomposition of the integral will be at speed 1 throughout. 
The product rule yields
\[|S^{\le 10}\Box_c  W_j| \le r^{-1}|S^{\le 5} \partial V_{j-1}||S^{\le 10} \partial V_{j-1}|.\]

Away from the speed $1$ light cone, we have
\begin{multline*}
  \| r S^{\le 10} \Box_c W_j \cdot S^{\le 10}\duuc W_j\|_{\ell^1_\tau
    \ell^1_{R\le \tau/2} L^1_tL^1_r(C^{1,R}_\tau)}
  \lesssim \|\la \uu\ra S^{\le 5} \partial V_{j-1}\|_{\ell^\infty_R
    \ell^\infty_{\tau\ge 2R} L^\infty_tL^\infty_r(C^{1,R}_\tau)}
\\\times  \|\la r\ra^{-\frac{1}{2}} S^{\le 10}\partial
  V_{j-1}\|_{\ell^\infty_R \ell^2_\tau L^2_tL^2_r(C^{1,R}_\tau)}
  \|S^{\le 10} \duuc W_j\|_{\ell^2_\tau \ell^2_R L^2_tL^2_r(C^{1,R}_\tau)}.
\end{multline*}
Here, we have used that $\la r\ra\lesssim \la \uu\ra$.  Moreover, the
factor of $\la \uu\ra^{-\frac{1}{2}}$ is used to control the remaining
$\ell^2_R$ summation.  As \eqref{ptwise_V_crt} gives that this is $\lesssim \varepsilon (VI)_{j-1} (XI)_j,$
\eqref{bdd_ih} shows that these terms are controlled by
\eqref{ItoVgoal} (without the logarithmic factor, in fact).

Near the speed $1$ light cone, we use the Schwarz inequality
to bound
\begin{multline*}
  \| r S^{\le 10} \Box_c W_j \cdot S^{\le 10}\duuc W_j\|_{\ell^1_\tau
    \ell^1_{U\le \tau/4} L^1_tL^1_r(C^{1,U}_\tau)}
  \lesssim \|\la u\ra S^{\le 5} \partial V_{j-1}\|_{\ell^\infty_U
    \ell^\infty_{\tau\ge 4U} L^\infty_tL^\infty_r(C^{1,U}_\tau)}
\\\times  \|\la r\ra^{-\frac{1}{2}} S^{\le 10}\partial
  V_{j-1}\|_{\ell^2_U \ell^2_\tau L^2_tL^2_r(C^{1,U}_\tau)}
  \|r^{\frac{1}{2}}\la u\ra^{-\frac{1}{2}}S^{\le 10} \duuc W_j\|_{\ell^\infty_U \ell^2_\tau L^2_tL^2_r(C^{1,U}_\tau)}.
\end{multline*}
The additional $\la
u\ra^{-\frac{1}{2}}$ allowed us to absorb a square summation in $U$.  Upon
applying \eqref{UtoR} and \eqref{ptwise_V_cut}, these final terms are
\[\lesssim \varepsilon (\log(2+T_\varepsilon))^{\frac{1}{2}}(VI)_{j-1}
  (X)_j.\]
Thus, \eqref{bdd_ih} gives that they are controlled by
\eqref{ItoVgoal}.  This completes the proof of \eqref{bdd_goal}.

\subsection{Convergence} We now show that the sequence $((V_j,W_j))$
converges by showing that it is Cauchy in an appropriate norm.  With this aim, we set
\begin{multline}
\label{A}
  A_j = 
  \|r^{-\frac{1}{4}}\la r\ra^{-\frac{1}{4}} \la u\ra^{\frac{1}{2}} S^{\le 7} \du
  (V_j-V_{j-1})\|_{\ell^\infty_R \ell^2_\tau L^2_tL^2_r(C^{1,R}_\tau)}  
  + \|\la u_c\ra^{-\frac{1}{2}}\la u\ra ^{\frac{1}{2}} S^{\le 7}
  \du (V_j-V_{j-1})\|_{\ell^\infty_{U_c}\ell^2_\tau L^2_tL^2_r(C^{c,U_c}_\tau)}
  \\+\|r^{-\frac{1}{4}}\la r\ra^{-\frac{1}{4}} \la\uu\ra^{\frac{1}{2}} S^{\le 7} \duu
  (V_j-V_{j-1})\|_{\ell^\infty_R \ell^2_\tau L^2_tL^2_r(C^{1,R}_\tau)}  
  + \|\la u_c\ra^{-\frac{1}{2}}\la \uu\ra^{\frac{1}{2}} S^{\le 7}
  \duu (V_j-V_{j-1})\|_{\ell^\infty_{U_c}\ell^2_\tau L^2_tL^2_r(C^{c,U_c}_\tau)}
  \\ +
  \|\la u\ra^{-\frac{1}{2}}\la \uu\ra^{\frac{1}{2}} S^{\le 7}
  \duu (V_j-V_{j-1})\|_{\ell^\infty_{U}\ell^2_\tau L^2_tL^2_r(C^{1,U}_\tau)}
  +\|r^{-\frac{1}{4}}\la r\ra^{-\frac{1}{4}} S^{\le 10} \partial
  (V_j-V_{j-1})\|_{\ell^\infty_R \ell^2_\tau L^2_tL^2_r(C^{1,R}_\tau)}  
  \\+ \|\la u_c\ra^{-\frac{1}{2}} S^{\le 10}
  \partial (V_j-V_{j-1})\|_{\ell^\infty_{U_c}\ell^2_\tau L^2_tL^2_r(C^{c,U_c}_\tau)}
  +\|\la u\ra^{-\frac{1}{2}}S^{\le 10}
  \duu (V_j-V_{j-1})\|_{\ell^\infty_{U}\ell^2_\tau L^2_tL^2_r(C^{1,U}_\tau)}
  \\
  +\|r^{\frac{1}{2}}\la u\ra^{-\frac{1}{2}} S^{\le 10}
  \duuc (W_j-W_{j-1})\|_{\ell^\infty_{U}\ell^2_\tau L^2_tL^2_r(C^{1,U}_\tau)}
  +\|r^{-\frac{1}{2}}\la u\ra^{-\frac{1}{2}} S^{\le 10}
  (W_j-W_{j-1})\|_{\ell^\infty_{U}\ell^2_\tau L^2_tL^2_r(C^{1,U}_\tau)}
  \\ +\|S^{\le 10} \duuc (W_j-W_{j-1})\|_{\ell^2_R \ell^2_\tau L^2_tL^2_r(C^{1,R}_\tau)}  
  +\|r^{-1} S^{\le 10}  (W_j-W_{j-1})\|_{\ell^2_R \ell^2_\tau L^2_tL^2_r(C^{c,R}_\tau)},
\end{multline}
and we will show that, for each $j$,
\begin{equation}
    A_j\leq \frac{1}{2}A_{j-1}.
\end{equation}
We note that
\[\Box(V_j-V_{j-1}) = r^{-1} \Bigl((W_{j-1}-W_{j-2})\partial_t V_{j-1} + W_{j-2}\partial_t(V_{j-1}- V_{j-2})\Bigr)\]
and
\[\Box_c(W_j-W_{j-1}) =  r^{-1} \partial_t(V_{j-1}- V_{j-2})(\partial_t V_{j-1}+\partial_t V_{j-2}).\]
For this system, we call upon the same arguments as in the proof
of \eqref{bdd_goal}.  Doing so yields
\[A_j\lc A_{j-1}(M_{j-1}+M_{j-2}),\]
and applying \eqref{bdd_goal}
with $\varepsilon$ sufficiently small completes the proof.

\bibliography{exterior}

\end{document}